\documentclass[10pt]{article}

 \usepackage{geometry}
\usepackage{amsmath,amsthm,amsfonts,amssymb}
\usepackage{appendix}
\usepackage{verbatim}
\usepackage{paralist,graphics,epsfig,graphicx,epstopdf,mathrsfs}
\usepackage[colorlinks=true]{hyperref}
\hypersetup{urlcolor=blue, citecolor=red}

\newtheorem{theorem}{Theorem}[section]
\newtheorem{corollary}[theorem]{Corollary}

\newtheorem{assumption}[theorem]{Assumption}
\newtheorem{proposition}[theorem]{Proposition}

\theoremstyle{definition}
\newtheorem{definition}[theorem]{Definition}
\numberwithin{equation}{section}
\newtheorem{remark}{Remark}

\def\R{\mathbb{R}}
\def\F{\mathscr{F}}
\def\Ft{\lbrace \mathscr{F}_t\rbrace_{t\geq 0}}
\def\CIP{\mathcal{C}_{[-\tau,0]}^{+}}
\def\MIP{\mathcal{M}_{[-\tau,0]}^+}

\begin{document}

	\title{On invariant measures and the asymptotic behavior of a stochastic delayed SIRS epidemic model}
	\author{\textsc{Xiaoming Fu}\thanks{This work is supported by Chinese Scholarship Council.\;\texttt{(xiaoming.fu@u-bordeaux.fr)}}
		\\
	{\small \textit{IMB, UMR 5251, University of Bordeaux, 33400 Talence, France}} \\
}
	\date{}
	\maketitle
	
	\begin{abstract}
		In this paper, we consider a stochastic epidemic model with time delay and general incidence rate. We first prove the existence and uniqueness of the global positive solution. By using the Krylov-Bogoliubov method, we obtain the existence of invariant measures. Furthermore, we study a special case where the incidence rate is bilinear with distributed time delay. When the basic reproduction number $ \mathcal{R}_0<1 $, the analysis of the asymptotic behavior around the disease-free equilibrium $ E_0 $ is provided while when $ \mathcal{R}_0>1 $, we prove that the invariant measure is unique and ergodic. The numerical simulations also validate our analytical results.
	\end{abstract}
	\textbf{Key words: } Stochastic delayed SIRS model; General incidence rate;  Invariant measure; Asymptotic behavior \\



\section{Introduction}

Epidemics are commonly studied by using deterministic
compartmental models where the population is divided into several classes, namely susceptible, infected, and recovered groups. Beretta et al. \cite{Beretta95} studied a vector-borne SIR model with distributed delay. Zhen et al. extended this model by allowing the loss of immunity and showed stability results for the following SIRS model (see \cite{ZMH06} for the derivation of the model)
  	\begin{equation}\label{Zhen}
 	\left\{\begin{array}{ll}
 	\dot{S}(t)=\lambda-\beta S(t)\int_{0}^{h}f(s)I(t-s)\mathrm{d}s-\mu S(t)+\eta R(t) ,\\[0.2cm]
 	\dot{I}(t)=\beta S(t)\int_{0}^{h}f(s)I(t-s)\mathrm{d}s-(\mu+\delta+\gamma)I(t),\\[0.2cm]
 	\dot{R}(t)=\gamma I(t)-(\mu+\eta) R(t),
 	\end{array}	\right.
 	\end{equation}
 where $ h>0 $ is the time delay, $ \lambda $ is the recruitment rate of the population, $ \mu $ is the natural death rate of the population and $ \delta $ is the rate of the additional death due to disease. Moreover, $ \gamma $ is the recovery rate of infected individuals while $ \eta $ is the rate of loss of immunity. $ \beta $ represents the disease transmission coefficient. Finally, $ f(t) $ represents the fraction of vector population in which the total time taken to become infectious is $ t $ and $ \int_0^h f=1 $. 
 
In this work, we generalize the model by setting the incidence rate as $ \beta S(t)H(I_{t}) $, where $ H:C([-\tau,0];\mathbb{R}) \to \R $ is a functional satisfying certain assumptions. We also introduce stochastic effects as in \cite{ImhofWalcher05, Jiang11} where we assume the natural death rate $ \mu $ fluctuate around some average value due to the randomness in the environment. In such a way, $ \mu $ becomes a random variable $ \tilde{\mu} $, i.e.,
\begin{equation*}
\tilde{\mu}\mathrm{d}t=\mu \mathrm{d}t -\sigma \mathrm{d}W(t),
\end{equation*}
here $ \sigma>0 $ represents the intensity of the noise and $ W(t) $ is a scalar Brownian motion. Therefore, our model can be written as follows:
\begin{equation}\label{eq: SFDE}
\left\{\begin{array}{ll}
\mathrm{d}S(t)=\big(\lambda-\mu S(t)-\beta S(t)H(I_{t})+\eta R(t)\big)\mathrm{d}t+ \sigma S(t)\mathrm{d}W(t),\\[0.2cm]
\mathrm{d}I(t)=\big(\beta  S(t)H(I_{t})-(\mu+\gamma+\delta)I(t)\big)\mathrm{d}t+\sigma I(t)\mathrm{d}W(t),\\[0.2cm]
\mathrm{d}R(t)=\big(\gamma I(t)-(\mu + \eta )R(t)\big)\mathrm{d}t+\sigma R(t)\mathrm{d}W(t).
\end{array}	\right.
\end{equation} 

Environmental noises have a critical influence on the development of an epidemic. In the biological models, there are different ways to introduce the randomness in population systems. Tornatore et al. \cite{Tornatore05} considered some stochastic environmental factors acting simultaneously on the transmission coefficient $ \beta $. In their work, they studied the threshold effect for the stochastic SIR model and gave sufficient conditions for the disease-free equilibrium to be globally asymptotically stable without time delay and stable in probability with distributed time delay. In the work of Grey et al. \cite{Grey11}, they considered the same type of stochastic environmental impact on the transmission coefficient $ \beta $. In such a way, they established conditions for extinction and persistence of a stochastic SIS model.    We also refer the reader to \cite{Beretta98,Cai15,Hattaf,Lahrouz13,Li15,Liu17,Liu17_2,Lu09,Yang14,Zhao15} and the references therein for more models regarding the persistence and the extinction of populations in a stochastic environment. 

One approach to study the asymptotic behavior of the stochastic solution was considered by Jiang et al. \cite{Jiang11}, Liu et al. \cite{Liu15} and Yang et al. \cite{Yang12}. In their papers, they investigated the asymptotic behavior around the disease-free and endemic equilibrium by measuring the mean value of the oscillation between the solution and the equilibrium, which can be small if the diffusion coefficients are sufficiently small. Inspired by these works, we obtain the similar asymptotic results in this paper.

In this work, we also focus on the existence of invariant measures for system \eqref{eq: SFDE}. Yang et al. \cite{Yang12} considered the ergodicity property of a stochastic SIRS epidemic model with bilinear incidence rate. Cai et al. \cite{Cai17} studied a SIRS model epidemic with nonlinear incidence rate and provided analytic results regarding the invariant density of the solution. In addition, Rudnicki \cite{Rudnicki03} studied the existence of an invariant density for a predator-prey type stochastic system. 
For the study of invariant measures of stochastic functional differential equations (SFDE), Es-Sarhir \cite{Es-Sarhir10} considered a SFDE with super-linear drift term, while Kinnally and Williams \cite{KW10} considered a model with positivity constraints. We also refer the reader to Liu et al. \cite{Liu17_3} for more details on the stationary distribution of stochastic delayed epidemic models.

The structure of the paper is as follows. In Section 2, we introduce the notations and illustrate the main results. In Section 3, we prove the existence and uniqueness of the non-explosive positive solution of model \eqref{eq: SFDE} without using Lyapunov functionals.
Section 4 is focused on giving a sufficient condition for the existence of invariant measures for our model \eqref{eq: SFDE} and Section 5 is devoted to the asymptotic behavior of the solution and the ergodicity of the unique invariant measure. In the end, we present numerical simulations in Section 6 which support our results.

\section{Preliminary and main results}

Throughout this paper,  we let  $ (\Omega,\F,\Ft,P) $	be a complete probability space with a filtration $ \lbrace \F_t\rbrace_{t\geq 0} $ satisfying the usual conditions (i.e., $ \F_0 $ contains $ P $-null sets of $ \F $ and $ \F_{t+}:=\cap_{s>t}\F_s=\F_t $) and we let $\lbrace W(t) \rbrace_{t\geq 0}$ be a scalar Brownian motion defined on the probability space. In addition, for $ \tau>0 $ we define $ \mathcal C_{[-\tau,0]}:=C([-\tau,0];\mathbb{R}^n)$ the space of continuous functions from $ [-\tau,0] $ to $ \R^n $ endowed with the supremum norm, $ \mathcal M_{[-\tau,0]}:=\mathcal{B}(\mathcal C_{[-\tau,0]}) $ the associated Borel $ \sigma $-algebra. Similarly, we set $ \mathcal{C}_{[-\tau,\infty)}:=C_{loc}([-\tau,\infty);\mathbb{R}^n)$ the space of continuous functions from $ [-\tau,\infty) $ to $ \R^n $ with the topology of uniform  convergence on compact sets and let  $\mathcal M_{[-\tau,\infty)}:=\mathcal{B}(\mathcal C_{[-\tau,\infty)})$. \\
 
For any $ x \in \mathcal{C}_{[-\tau,\infty)}  $,  $ x_t $ denotes the segment process of $ x $ given by
\[ x_t(\theta)=x(t+\theta),\ \ \theta\in [-\tau,0],\ t\geq 0. \]
For any vector $ v\in\R^n $, we define $ |v|:=(\sum_{i=1}^n v_i^2)^{1/2} $ as the Euclidean norm. For any $ x\in\mathcal C_{[-\tau,0]} $, we define $$ \Vert x\Vert:=\sup_{\theta\in [-\tau,0]}|x(\theta)|. $$ In this paper, we always assume that the initial value $ \xi=(\xi_1,\xi_2,\xi_3) \in \CIP\cap \F_0$ which is a $ C([-\tau,0];\mathbb{R}_+^3) $-valued random variable and is $\F_0 $-measurable. 

For a general $ n $-dimensional stochastic functional differential equation
\begin{equation}\label{eq-general-SFDE}
\mathrm{d}X(t)=b(X_t)\mathrm{d}t +\sigma(X)\mathrm{d}W(t),
\end{equation}
where $ b(\cdot,t) $ is from $C_{[-\tau,0]}  $ to $ \R^n $, $ \sigma(\cdot,t) $ is from $ \R^n $ to $ \R^{n\times m} $, and $ W(t) $ is an $ m $-dimensional Brownian motion on $ (\Omega,\F,\Ft,P) $.
We define the differential operator $ L $ as 
\[ L=\frac{\partial}{\partial t} + \sum_{i=1}^{n} b_i(X_t) \frac{\partial}{\partial X_i} +\frac{1}{2} \sum_{i,j=1}^{n}[\sigma^T(X)\sigma(X)]\frac{\partial^2}{\partial X_i\partial X_j}.\]
For any $ V\in C^{2,1}\left(\R^n\times [0,\infty)\right) $ which is twice continuously differentiable in $ x $ and once in $ t $, one has 
\[  LV(X(t),t)=V_t(X(t),t)+V_x(X(t),t)b(X_t)+\frac{1}{2}\mathrm{trace}\left[\sigma^T(X)V_{xx}(X(t),t)\sigma(X)\right], \]
where $ V_t(X,t)=\frac{\partial V}{\partial t},\,V_x(X,t)=\left(\frac{\partial V}{\partial X_1},\dots,\frac{\partial V}{\partial X_n}\right) $ and $ V_{xx}(X,t)=\left(\frac{\partial^2 V}{\partial X_i \partial X_j}\right)_{n\times n} $. 
By the It\^{o} formula \cite{Maoxuerong}, if $ X(t)\in \R^n $, then 
\[ \mathrm{d}V(X(t),t)=LV(X(t),t)\mathrm{d}t+V_x(X(t),t)\sigma(X)\mathrm{d}W(t). \]
The diffusion matrix is defined as follows:
\[ A(X)=\left(a_{ij}(X)\right),\quad a_{ij}(X)=\sum_{l=1}^{n}\sigma_{il}(X)\sigma_{lj}(X). \]

The following proposition is needed for the uniqueness and the ergodic property of the invariant measure in our proof. 
\begin{proposition}\cite{Kha11}\label{pro-ergodic}
	There exists a bounded open domain $ U \subset \R^n $ with smooth boundary $ \partial U $, which has the following properties: 
	\begin{itemize}
		\item[(i)] In the domain $ U $ and some neighborhood thereof, the smallest eigenvalue of the diffusion matrix $ A(X) $ is bounded away from zero.
		\item[(ii)] If $ x\in \R^n\backslash U $, the mean time $ \tau $ at which a path issuing from $ x $ reaches the set $ U $ is finite, and $ \sup_{x\in K}E^x \tau <\infty $ for every compact subset $ K\subset \R^n $. 
	\end{itemize}
	If the above assumptions hold, then the Markov process $ X(t) $ with initial value $ X_0\in \R^n $ has a unique stationary distribution $ \pi(\cdot) $. Moreover, if $ f(\cdot) $ is a function integrable with respect to the measure $ \pi $, then 
	\[ P \left\lbrace \lim_{T\to \infty} \frac{1}{T} \int_0^T f\left(X^x(t)\right)\mathrm{d}t=\int_{R^n} f(x)\pi (\mathrm{d}x) \right\rbrace=1,\;\forall x\in \R^n. \]
\end{proposition}
\begin{remark}\label{rem2}
	To prove condition (i), it is sufficient to verify that there exists a positive number $ \delta $ such that the diffusion matrix satisfies $ \sum_{i,j=1}^{n}a_{ij}(x)\xi_i\xi_j> \delta |\xi|^2,\,x\in U,\,\xi\in \R^n $ (see \cite{Gard88}). This is the case for equation \eqref{eq: SFDE} if we choose any $ U $ such that the closure $ \overline{U}\subset \R_+^3 $. To verify condition (ii) in our case, one sufficient condition is to prove that there exists a non-negative $ C^2 $ function $ V :\R_+^3\to \R $ and a neighborhood $ U $ such that for some $ \kappa>0,\,LV(x)<-\kappa,\,x\in \R_+^3 \backslash U $ (see e.g. \cite{ZhuYin07}).
\end{remark}
In the following, we fix the dimension $ n=3 $ and define 
\[ \CIP:=C([-\tau,0];\mathbb{R}_+^3),\quad \MIP:=\mathcal{B}(\CIP). \]
\begin{assumption}\label{ass-A}
	The functional $ H: C([-\tau,0];\mathbb{R}) \to \R $ satisfies the following conditions: for any $ \phi,\varphi \in   C({[-\tau,0]};\R )$,
	\begin{align*}
	\left.
	\begin{array}{ll}
	(i)&\,|H(\phi)|\leq c(1+\Vert \phi\Vert),\\
	(ii)&\,|H(\phi)-H(\varphi)|\leq L_m\Vert \phi-\varphi \Vert,\ \text{for any}\ \Vert \phi\Vert, \Vert \varphi\Vert\leq m, \\
	(iii)&\,H(\phi)>0,\ \text{for any}\  \phi>0\ a.e.\ \text{on}\  [-\tau,0],
	\end{array}\right.
	\end{align*}
	where $ c$ is a positive constant and $ L_m $ is the Lipschitz constant on the bounded domain.
\end{assumption}
\begin{remark}
	Assumption \ref{ass-A} can be verified by various types of nonlinear
	transmission functions. For example, the distributed delay type functional, as in model \eqref{Zhen}, satisfies Assumption \ref{ass-A}. The general saturation
	incidence type functional
	\begin{equation*}
	H(I_t):=\frac{I(t-\tau)}{1+\alpha I(t-\tau)^q},\quad \alpha, q\in \R_+ ,
	\end{equation*}
	also verifies the conditions in Assumption \ref{ass-A}.
\end{remark}
The main results of this paper are as follows: Theorem \ref{thm: positive} ensures the well-posedness of the global positive solution under Assumption \ref{ass-A}. Under the same assumption, Theorem \ref{thm main result 1}  shows  that there exists an invariant measure for system \eqref{eq: SFDE}. 

In Section 5, we set $ H(\phi) $ as the distributed delay type functional, i.e.,
\begin{equation*}
H(\phi)=\int_{0}^{\tau}f(s)\phi(-s)\mathrm{d}s,\ \text{for any}\ \phi \in \CIP.
\end{equation*}
Theorem \ref{thm disease free} shows that when $ \mathcal{R}_0=\frac{\beta\lambda}{\mu(\mu+\gamma+\delta)}<1 $ and $ \mu $ satisfies certain conditions, we have an asymptotic estimation, where the limit
\begin{equation*}
\limsup_{t\to \infty}\frac{1}{t}\int_{0}^{t}E\left[\left(S(s)-\frac{\lambda}{\mu}\right)^2+I(s)+R(s)\right]ds
\end{equation*}
can be controlled by the noise coefficient $ \sigma $. 

Furthermore, for the case when $ \mathcal{R}_0>1 $ and if, in addition, we have $ \mu S^*-\eta R^*>0 $ (see \eqref{eq5.2} for the definitions) and the noise coefficient $ \sigma $ is small enough, system \eqref{eq: SFDE} has a unique invariant measure and it is ergodic.

\section{Well-posedness of the global positive solution}
As a biological model, we are interested in positive solutions. In order to ensure that a solution of stochastic functional differential equation is unique and does not blow up in finite time, the drift coefficient $ b $ and diffusion coefficient $ \sigma $ in \eqref{eq-general-SFDE} generally need to satisfy linear growth conditions \cite{Maoxuerong}. However, for system \eqref{eq: SFDE}, we do not have linear growth conditions on the drift and the diffusion terms. Thus we give a new method to prove the existence and uniqueness of the global positive solution.

\begin{theorem}\label{thm: positive}
	Let Assumption \ref{ass-A} be satisfied. 
	There exists a unique positive solution 
	\[ X(t)=(S(t),I(t),R(t))\in \R_+^3,\quad a.s. \]
	 to the equation \eqref{eq: SFDE} on $t\in  [0,\tau_e) $ for any initial value $ \xi=(\xi_1,\xi_2,\xi_3) \in\CIP\cap \F_0 $ where the random variable $ \tau_e$ is the explosion time.
\end{theorem}
\begin{proof}
	Since the coefficients of system \eqref{eq: SFDE} are locally Lipschitz continuous by Assumption \ref{ass-A}, for any given initial value $ \xi \in \CIP\cap \F_0 $, there exists a unique local solution $ X(t) $ on $ [-\tau,\tau_e) $, where  
	\[ \tau_e=\sup \left\lbrace t\geq 0: \sup_{s\in[0,t]}|X(s)|<\infty \right\rbrace \]
	is the explosion time (see Mao \cite{Maoxuerong} or Ikeda et al. \cite{Ikeda}). 
	Let us define the stopping time
	\[ \tau_S:=\inf\lbrace t\in [0,\tau_e)  : S(t) \leq 0\rbrace. \]
	Similarly, one can define $ \tau_I,\tau_R $ for the infected group and the recovered group respectively.
	Since $ R(t) $ satisfies the linear stochastic differential equation
	\[ \mathrm{d}R(t)= \left(\gamma I(t)-(\mu + \eta )R(t)\right)\mathrm{d}t+\sigma R(t)\mathrm{d}W(t), \]
	where $ I(t) $ is an $ \Ft $-adapted and almost surely locally bounded process. Thus, by \cite[Chap. 5.6.C]{karatzas12} one has
	\begin{equation*}
	R(t)=Z_R(t)\left(R(0)+\int_{0}^{t}\frac{\gamma I(u)}{Z_R(u)}\mathrm{d}u\right),\ \ t\in[0,\tau_e),
	\end{equation*}
	where 
	\[ Z_R(t)=\exp\left[-(\mu+\eta+1/2\sigma^2)t+\sigma W(t)\right]>0,\,a.s. \]	
	Thus, we have $ \tau_R \geq \tau_I $ almost surely. Since $ R(t)\geq 0\;a.s. $ on $ [0,\tau_R) $ and $ \tau_I\leq \tau_R\,a.s. $, we can see from \eqref{eq: SFDE} that
	\begin{equation*}
	\mathrm{d}S(t)\geq [\lambda-(\mu+\beta H(I_t))S(t)]\mathrm{d}t+\sigma S(t)\mathrm{d}W(t),\ \ t\in[0,\tau_I).
	\end{equation*}
	If we denote $ \underline{S}(t) $ to be the solution of 
	\[ d\underline{S}(t) = [\lambda-(\mu+\beta H(I_t))\underline{S}(t)]\mathrm{d}t+\sigma \underline{S}(t)\mathrm{d}W(t),\ \ t\in[0,\tau_I), \]
	with $ \underline{S}(0)=S(0) $.
	By the comparison theorem in \cite{Ikeda}, we have
	\begin{equation*}
	S(t)\geq \underline{S}(t)= Z_S(t)\left(S(0)+\int_0^t \frac{\lambda}{Z_S(u)}\mathrm{d}u\right),\ \ t\in[0,\tau_I),
	\end{equation*}
	where  
	\[ Z_S(t)=\exp\left[-(\mu+\sigma^2/2)t-\int_{0}^{t}\beta H(I_u)\mathrm{d}u+\sigma W(t)\right]>0,\,a.s. \]
	Therefore, we deduce that $ \tau_I\leq \min\lbrace \tau_S,\tau_R\rbrace $ almost surely. For the infected group, we have
	\begin{equation*}
	I(t)=Z_I(t)\left(I(0)+\int_{0}^{t}\frac{\beta S(u)H(I_u)}{Z_I(u)}\mathrm{d}u\right),\ \ t\in[0,\tau_e),
	\end{equation*}
	where $$ Z_I(t)=\exp\left[-(\mu+\gamma+\delta+\sigma^2/2)t+\sigma W(t)\right]. $$ 
	
	Next, we claim that $  \tau_e\leq \tau_I,$ almost surely. If this is true, then $ \tau_e\leq \min\lbrace \tau_S,\tau_I,\tau_R \rbrace $ almost surely, the result follows. 
	We argue by contradiction. Suppose that there exists a set $ E \in \mathcal B(\Omega) $ with $ P(E)>0 $ and for any $ \omega \in E $, one has $ \tau_e(\omega)>\tau_I(\omega) $. Since $ \tau_S\geq \tau_I $ almost surely, we can choose an $ \omega_0\in E $ such that 
	$ \tau_e(\omega_0)>\tau_I(\omega_0) $ and $ \tau_S(\omega_0)\geq \tau_I(\omega_0) $. Since 
	\begin{equation}\label{eq3.1}
	I(t,\omega_0)>0, \; \forall t\in [0,\tau_I(\omega_0))\; \text{ and }\; I(\tau_I(\omega_0),\omega_0)=0, 
	\end{equation}
	this yields 
	\begin{equation}\label{eq3.2}
		0=I(\tau_I(\omega_0),\omega_0)=Z_I(\tau_I(\omega_0),\omega_0)\left(I(0,\omega_0)+\int_{0}^{\tau_I(\omega_0)}\frac{\beta S(u,\omega_0)H(I_u(\cdot,\omega_0))}{Z_I(u,\omega_0)}\mathrm{d}u\right).
	\end{equation}
	However, from Assumption \ref{ass-A} (iii) and \eqref{eq3.1}, we obtain
	\[ H(I_u(\cdot,\omega_0))>0,\;\quad   \forall u\in [0,\tau_I(\omega_0)). \]
	Moreover,  $ \tau_S(\omega_0)\geq \tau_I(\omega_0) $ yields
	\[ S(u,\omega_0)>0,\quad   \forall u\in [0,\tau_I(\omega_0)).\]
	Thus, the right hand side of \eqref{eq3.2} is strictly positive which is a contradiction. Hence, we must have $ \tau_e\leq \tau_I $ almost surely.
\end{proof}

\begin{corollary}\label{cor-positive+global}
		Let Assumption \ref{ass-A} be satisfied. Then  for any initial value $ \xi=(\xi_1,\xi_2,\xi_3)\in \CIP \cap \F_0$, there exists a unique positive solution $ X(t)=(S(t),I(t),R(t)) $ to the system \eqref{eq: SFDE} which does not blow up in finite time.
\end{corollary}
\begin{proof}
	By Theorem \ref{thm: positive}, we have $ \max\lbrace S(t),I(t),R(t) \rbrace\leq N(t),a.s.$ on $ [0,\tau_e) $, where $ N(t)=S(t)+I(t)+R(t) $. Moreover,
	\begin{align*}
	\mathrm{d}N(t)=&\left(\lambda-\mu N(t)-\delta I(t) \right)\mathrm{d}t + \sigma N(t)\mathrm{d}W(t)\\
	\leq&\left(\lambda-\mu N(t)\right)\mathrm{d}t+ \sigma N(t)\mathrm{d}W(t),\ \ t\in[0,\tau_e) .
	\end{align*}
	We denote by $ \tilde{N}(t) $ the solution of the following SDE with the same initial value  $ \xi\in \CIP\cap \F_0 $:
	\begin{equation*}
	\mathrm{d}\tilde{N}(t)=\left(\lambda-\mu \tilde{N}(t)\right)+ \sigma \tilde{N}(t)\mathrm{d}W(t).
	\end{equation*}
	Obviously, $ \tilde{N}(t) $ is a geometric Brownian motion and will not explode in finite time. Therefore, by the comparison theorem \cite{Ikeda}, we have $ 0\leq N(t)\leq \tilde{N}(t) <\infty $ on $ [0,\infty) $ almost surely.
\end{proof}

\section{Existence of invariant measures}

\subsection{A sufficient condition for the existence of invariant measures}

\begin{proposition}\cite[Propostion 2.1.2]{KW10}
	Let Assumption \ref{ass-A} be satisfied. From Theorem \ref{thm: positive} and Corollary \ref{cor-positive+global}, there exists a unique positive solution $ X^{x}(t)=(S(t),I(t),R(t)) $ to system \eqref{eq: SFDE} for any given $ \CIP $-valued initial condition $ X_0=x\in\F_0 $. Then the associated family of transition functions $ \lbrace P_t(\cdot,\cdot) \rbrace_{t\geq 0} $  of the segment process $ X_t^x $ defined by
	\begin{equation}\label{eq def of P_t}
	P_t(x,\Lambda):=P^x(X_t^x \in \Lambda), t\geq 0,\ \text{ for all }\ (x,\Lambda)\in\ \CIP \times \MIP
	\end{equation}
	is Markovian and Feller continuous.
\end{proposition}
\begin{remark}
		This proposition is obtained by several results of \cite{kinnally09} and it shows that the transition functions of the segment process $ X_t $ instead of $ X(t) $ has Markov property and Feller continuity.
		The main idea of the proof of Feller continuity is as follows: for a given sequence of $ \lbrace x_n \rbrace \subset \CIP $ with $ x_n\to x\in \CIP $ as $ n\to \infty $, let $ P^{x_n} $ be the distribution of the solution $ X^{x_n} $ to \eqref{eq: SFDE} satisfying the initial condition $ X_0^{x_n}=x_n $. The existence of the global solution is guaranteed by Theorem \ref{thm: positive} and Corollary \ref{cor-positive+global}. Furthermore, we can show that $ \lbrace P^{x_n}(X^{x^n}\in\cdot) \rbrace_{n \geq 0} $ on $ (\mathcal C_{[-\tau,\infty)},\mathcal{M}_{[-\tau,\infty)}) $ is tight. Let $ Q $ be any weak limit point of the sequence $ \lbrace P^{x_n}\rbrace_{n \geq 0} $, then we can prove $ Q $ is the distribution of the solution $ X^x $ to system \eqref{eq: SFDE} satisfying the initial condition $ X_0^x=x $. The Markov property is a consequence of the uniqueness of the solution. 
\end{remark}

\begin{definition}
	Let $ \lbrace P_t \rbrace_{t\geq 0} $ be a Markovian semigroup on  $ \left(\CIP,\MIP\right) $. A probability measure $ \mu $ on  $ \left(\CIP,\MIP\right) $ is called \textit{invariant measure} of $ \lbrace P_t \rbrace_{t\geq 0} $ if 
	\begin{equation*}
	\int_{E}P_t(y,\Lambda)\mu(dy)=\mu(\Lambda),\quad \text{ for all } t\geq 0\text{ and }\,\Lambda\in \MIP.
	\end{equation*}
\end{definition}
Given $ x \in \CIP $ and $ T>0 $, we define a set of probability measures $\lbrace Q_T^{x} \rbrace_{T\geq 0}$ on $\left(\CIP, \MIP\right)$ by:
\begin{equation*}
Q_T^{x}(\Lambda):=\frac{1}{T}\int_{0}^{T}P_t(x,\Lambda)\mathrm{d}t,\ \text{ for all }\ \Lambda\in \MIP,
\end{equation*}
where the set of probability measures is called the\textit{ Krylov-Bogoliubov measures}  associated with the transition functions $ \lbrace P_t(\cdot,\cdot) \rbrace_{t\geq 0} $  of the stochastic functional differential equation.

To demonstrate the existence of an invariant measure for a Feller continuous process,  one  typical method is to show the weak convergence of a sequence of the Krylov-Bogoliubov measures \cite{DPG;ZJ} by using the tightness criterion of probability measures on the continuous function space \cite{BP}. It is well known (see e.g. \cite[Theorem 3.1.1]{KW10}) that one sufficient condition for the tightness of Krylov-Bogoliubov measures is the uniform boundedness of the segment process, i.e.,
	\begin{equation*}\label{C1}
	\ \sup_{t\geq 0}E\Vert X_t\Vert<\infty,\tag{\textbf{C}}
	\end{equation*}
where $ X(t)=(S(t),I(t),R(t)) $, we denote the above condition by \eqref{C1}.

\subsection{Invariant measure for the stochastic delayed SIRS model}

For our specific epidemic model \eqref{eq: SFDE}, we need to verify the condition \eqref{C1} to prove the existence of an invariant measure. 
Before we begin the proof, we present the proposition from \cite[Theorem 4.]{Scheutzow13}.

\begin{proposition}\label{pro:mart.ineq}
	For each $ p\in (0,1) $, let $ Z $ and $ H $ be non-negative, $ \Ft $ adapted processes (i.e., $Z(t),H(t)\in\F_t,t\geq 0  $) with continuous paths. Assume that $ \varphi $ is a non-negative deterministic function. Let  $ M(t), t\geq 0 $ be a continuous local martingale starting at $ M(0)=0 $. If
	\begin{equation*}
		Z(t)\leq \int_0^t\varphi (s)Z(s)\mathrm{d}s+M(t)+H(t)
	\end{equation*}
	holds for all $ t\geq 0 $, then we have
	\begin{equation*}
	E\left(\sup_{s\in[0,t]}Z^p(s)\right)\leq c_p\exp\left( p\int_{0}^{t}\varphi(s)\mathrm{d}s \right) E\left(\sup_{s\in [0,t]}H^p(s)\right)
	\end{equation*}
	holds for some constant $ c_p $.
\end{proposition}

\begin{theorem}\label{thm main result 1}
	Suppose Assumption \ref{ass-A} is satisfied and let us denote $ X(t)=(S(t),I(t),R(t))  $ the solution of system \eqref{eq: SFDE} with initial value $ \xi \in \CIP \cap \F_0 $. Then if in addition $ \sum_{i=1}^3E\xi_i(0)<\infty $,  we have
	\begin{equation*}
	\sup_{t\geq 0}E\Vert X_t^{\xi}\Vert<\infty,
	\end{equation*}
	Therefore, system \eqref{eq: SFDE} admits an invariant measure.
\end{theorem}

\begin{proof}
	We use the same notations as in Corollary \ref{cor-positive+global} and denote $ \tilde{N}(t) $ to be the solution of the following SDE with initial value $ \tilde{N}(0)=\sum_{i=1}^3\xi_i(0)$,
	\begin{equation}\label{Ntilde}
	\mathrm{d}\tilde{N}(t)=\left(\lambda-\mu \tilde{N}(t)\right)\mathrm{d}t+ \sigma \tilde{N}(t)\mathrm{d}W(t).
	\end{equation}
	By the same argument as in Corollary \ref{cor-positive+global}, we have $ N(t)\leq \tilde{N}(t) $ for any $ t\in [0,\infty) $ almost surely. By the It\^{o} formula, we have for any $ q>1 $
	\begin{align*}
	\mathrm{d}\tilde{N}^q(t)=\left(q\lambda\tilde{N}^{q-1}(t)-q\mu\tilde{N}^q(t)+\frac{\sigma^2}{2}q(q-1)\tilde{N}^q(t)\right)\mathrm{d}t+q\sigma\tilde{N}^q(t)\mathrm{d}W(t).
	\end{align*}
	 By Young's inequality $ a^{\frac{1}{q}}b^{\frac{q-1}{q}}\leq \frac{1}{q}a+\frac{q-1}{q}b ,$ we have 
	\[ q\lambda\tilde{N}^{q-1}(t)=\big(q\lambda^q\big)^{\frac{1}{q}}\left(q\tilde{N}(t)^q\right)^{\frac{q-1}{q}}\leq \lambda^q+(q-1)\tilde{N}^q(t). \] 
	Therefore we obtain
	\begin{align*}
		\mathrm{d}\tilde{N}^q(t)\leq \left(\lambda^q-\alpha\tilde{N}^q(t)\right)\mathrm{d}t+q\sigma\tilde{N}^q(t)\mathrm{d}W(t),
	\end{align*}
	where $$ \alpha:=q\mu-(q-1)-\frac{\sigma^2}{2}q(q-1).$$ 
	Since $ \mu>0 $, we choose $ q>1 $ but sufficiently close to $ 1 $ such that $ \alpha>0 $. For any $ t\geq  0 $, we define $ Y(t) $ to be the solution of the following linear stochastic differential equation
		\begin{equation*}
		dY(t)=\left(\lambda^q -\alpha Y(t)\right)\mathrm{d}t+q\sigma Y(t)\mathrm{d}W(t).
		\end{equation*}
		with $ Y(0)=N^q(0) $. 
	Thus, we can obtain
		\begin{equation}\label{eq process Y}
		Y(t)=e^{-\alpha t}N^q(0)+\int_{0}^{t} \lambda^q e^{-\alpha(t-s)} \mathrm{d}s+ \int_{0}^{t}q\sigma e^{-\alpha(t-s)}Y(s)\mathrm{d}W(s).
		\end{equation}
	 By the comparison result, one has $0 \leq N^q(t) \leq \tilde{N}^q(t) \leq Y(t),\, a.s.$ for any $ t\geq  0 $. 
	By setting $$ H(t):=N^q(0)+\int_{0}^{t} \lambda^q e^{\alpha s} \mathrm{d}s \;\text{ and }\; M(t):=\int_{0}^{t}q\sigma e^{\alpha s}Y(s)\mathrm{d}W(s),  $$ we derive from \eqref{eq process Y} that
	\begin{equation}\label{eq:N(t)}
	0\leq e^{\alpha t}N^q(t)\leq  e^{\alpha t}Y(t)= M(t)+H(t),\ \ t\geq 0.
	\end{equation}
	 Since $ H $ and $ M $ satisfy the assumptions of Proposition \ref{pro:mart.ineq},
	 by equation \eqref{eq:N(t)} and Proposition \ref{pro:mart.ineq} with $ \varphi =0 $, for each $ p\in (0,1) $, there exists a $ c_p(\geq 0) $ such that
	\begin{equation}\label{eq: moment bound almost }
	 E\left[\sup_{s\in[0,t]}\left(e^{\alpha s}N^q(s)\right)^p\right]\leq E\left[\sup_{s\in[0,t]}\left(e^{\alpha s}Y(s)\right)^p\right]\leq c_p E[\sup_{s\in [0,t]}H^p(s)], \ \ t\geq 0.
	\end{equation}
	Moreover, one has
	\begin{align*}
	e^{-\alpha t p}E[\sup_{s\in [0,t]}H^p(s)]= & E\left[\sup_{s\in[0,t]}\left(e^{-\alpha t}N^q(0)+\int_{0}^{s}\lambda^q e^{-\alpha(s-l)}\mathrm{d}l\right)^p\right]  \\
	\leq &E\left[\left(N^q(0)+\int_{0}^{t}\lambda^q e^{-\alpha(t-s)}\mathrm{d}s\right)^p\right],
	\end{align*}
	where the last inequality is due to the fact that the mapping $ s\mapsto \int_{0}^{s}\lambda^q e^{-\alpha(s-l)}\mathrm{d}l $ is monotone increasing.
	Thus multiplying both sides of \eqref{eq: moment bound almost } by $ e^{-\alpha t p} $ yields
	\begin{align*}
	 E\left[\sup_{s\in[0,t]}\left(e^{-\alpha(t-s) }N^q(s)\right)^p\right]\leq& c_p E\left[\left(N^q(0)+\int_{0}^{t}\lambda^q e^{-\alpha(t-s)}\mathrm{d}s\right)^p\right]\\
	 \leq& c_pE\left[N^{qp}(0)+\frac{\lambda^{qp}}{\alpha^p}\right],\ \ t\geq 0,
	\end{align*}
	where we used the inequality $ (a+b)^p\leq a^p+b^p $ for $ a,b>0 $ and $ 0<p<1 $.
	If we let $ p=q^{-1} $ and by our assumption $ E[N(0)]<\infty $, we have for some constant $ c=c(q,E[N(0)])\geq 0 $ that
	\begin{equation*}
	E\left[\sup_{s\in[0,t]} e^{-\frac{\alpha}{q} (t-s) }N(s)\right]\leq c.
	\end{equation*}
	Finally, for any $ t\geq0 $ and $ \tau>0 $, one has
	\begin{align*}
		E\left[\sup_{s\in[0,t]} e^{-\frac{\alpha}{q} (t-s) }N(s)\right] \geq  E\left[\sup_{s\in[t-\tau,t]}e^{-\frac{\alpha}{q} (t-s) }N(s) \right]
	 \geq e^{-\frac{\alpha}{q}\tau}  E\left[\sup_{s\in[t-\tau,t]}N(s)\right].
	\end{align*}
	Therefore, since $ N\geq 0 $ a.s., we proved the uniform moment bound of the total population, i.e.,
	\begin{equation*}
	\sup_{t\geq0}E\Vert N_t \Vert\leq e^{\frac{\alpha}{q}\tau }  c.
	\end{equation*}
	Since the solution $ (S(t),I(t),R(t))\in \R_+^3 ,\,a.s.$ for any $ t\in [0,\tau_e) $, we can see that $ |X(t)|^2=S^2(t)+I^2(t)+R^2(t)\leq N^2(t),\,t\geq 0,\,a.s.  $, thus
		\begin{equation*}
		\sup_{t\geq0}E\Vert X_t \Vert\leq e^{\frac{\alpha}{q}\tau }  c <\infty,
		\end{equation*}
	which proves \eqref{C1}, we conclude that the system \eqref{eq: SFDE} admits an invariant measure.
\end{proof}

\section{Asymptotic behavior around the disease-free equilibrium and the endemic equilibrium}
In this section, we study the asymptotic behavior of stochastic SIRS model \eqref{eq: SFDE}
where we fix $ H(\cdot) $ to be of the following distributed delay form
\begin{equation*}
H(\phi)=\int_{0}^{\tau}f(s)\phi(-s)\mathrm{d}s,\ \text{for any}\ \phi \in \mathcal C_{[-\tau,0]}.
\end{equation*}
For the corresponding deterministic SIRS system of \eqref{eq: SFDE} (i.e., when $ \sigma=0 $), if $ \mathcal{R}_0:=\frac{\beta\lambda}{\mu(\mu+\gamma+\delta)}\leq 1 $, then there exists a unique disease-free equilibrium 
\begin{equation}\label{eq5.1}
 E_0=\left(\frac{\lambda}{\mu},0,0\right) 
\end{equation}
and it is globally stable   (see \cite{Beretta95}). Moreover, if $ \mathcal{R}_0> 1 $ the deterministic system admits a unique interior equilibrium
\begin{equation}\label{eq5.2}
 E^*=\left(\frac{\gamma +\delta +\mu }{\beta },\frac{(\eta +\mu ) (\beta  \lambda -\mu  (\gamma +\delta +\mu ))}{\beta  (\gamma  \mu +(\delta +\mu ) (\eta +\mu ))},\frac{\gamma  (\beta  \lambda-\mu  (\gamma +\delta +\mu ))}{\beta  (\gamma  \mu +(\delta +\mu ) (\eta +\mu ))}\right)=:(S^*,I^*,R^*), 
\end{equation}
which is globally stable under certain conditions (see \cite{Enastu12,Nakata11,ZMH06}).
However, $ E_0$ and $E^* $ are no longer equilibria for stochastic system \eqref{eq: SFDE}. Thus we study the stochastic solutions around $ E_0 $ and $ E^* $.

\subsection{ Around $ E_0  $ disease-free equilibrium}

\begin{theorem}\label{thm disease free}
	Let Assumption \ref{ass-A} be satisfied. Suppose $ \mathcal{R}_0=\dfrac{\beta\lambda}{\mu(\mu+\gamma+\delta)}< 1 $ and 
	\begin{equation}\label{condition-disease-free}
	\mu>\max \left\lbrace \gamma +\delta +\frac{3}{2} \left(\eta +\sigma ^2\right),\frac{\gamma ^2+\gamma  \delta -\left(2 \eta -\sigma ^2\right) \left(\delta -\eta -\sigma ^2\right)}{2 \left(\gamma +\delta -\eta -\sigma ^2\right)} \right\rbrace,\quad \gamma+\delta-\eta-\sigma^2>0.
	\end{equation}
	Then for any given initial value $ \xi$ in $ \CIP\cap \F_0 $, the solution of equation \eqref{eq: SFDE} has the following property
	\begin{equation}\label{eq mean stable}
	\limsup_{t\to \infty}\frac{1}{t}\int_{0}^{t}E\left[\left(S(s)-\frac{\lambda}{\mu}\right)^2+I(s)+R(s)\right]\mathrm{d}s\leq \frac{ \lambda ^2 \sigma ^2}{2K \mu ^2}\frac{c_1+1}{c_1},
	\end{equation}
	where  $ K=\min\left\lbrace \dfrac{2\mu-\eta-\sigma^2}{4},\dfrac{\lambda(\mu+\eta)(\mu+\gamma+\delta)}{\mu+\gamma+\eta}(1-\mathcal{R}_0) \right\rbrace $ and $ c_1=\dfrac{2(\gamma +\delta +\eta +\sigma^2)}{2\mu-\eta- \sigma ^2} $.
\end{theorem}
\begin{proof}
	First, we change the variables by $$ u=S-\frac{\lambda}{\mu},v=I, w=R .$$ Then system \eqref{eq: SFDE} can be written as
	\begin{equation*}
	\left\{\begin{array}{ll}
	\mathrm{d}u(t)=\left(-\mu u(t)-\beta (u(t)+\frac{\lambda}{\mu})H(v_t)+\eta w(t)\right)\mathrm{d}t+ \sigma (u(t)+\frac{\lambda}{\mu})\mathrm{d}W(t),\\[0.2cm]
	\mathrm{d}v(t)=\left(\beta (u(t)+\frac{\lambda}{\mu})H(v_t)-(\mu+\gamma+\delta)v(t)\right)\mathrm{d}t+\sigma v(t)\mathrm{d}W(t),\\[0.2cm]
	\mathrm{d}w(t)=\left(\gamma v(t)-(\mu+\eta) w(t)\right)\mathrm{d}t+\sigma w(t)\mathrm{d}W(t),
	\end{array}	\right.
	\end{equation*}
	and by Theorem \ref{cor-positive+global}, one has $ u\in \R, v>0,w>0 $. We define the non-negative function
	\begin{equation*}
	V(u,v,w)=c_1 u^2+c_2 w^2+c_3 v+c_4 w+(u+v)^2+\frac{2 c_1 \beta  \lambda ^2}{\mu ^2}\int_{0}^{\tau}\int_{t-s}^{t}v(r)\mathrm{d}rf(s)\mathrm{d}s,
	\end{equation*}
	where $ c_i,\,i=1,2,3,4 $ are positive constants to be chosen later. 
	Then, by the It\^{o} formula, one has
	\begin{equation*}
	\mathrm{d}V=LV\mathrm{d}t+\sigma  \left(2(c_1+1) u^2 +2 v^2+2 c_1 w^2 +4 u v+ u\frac{2 (c_1+1) \lambda }{\mu }+v \left(c_3+\frac{2 \lambda }{\mu }\right)+c_4 w\right)\mathrm{d}W(t),
	\end{equation*}
	where
	\begin{align*}
	LV=&2c_1 u\left(-\mu u-\beta \left(u+\frac{\lambda}{\mu}\right)H(v_t)+\eta w\right)+c_1\sigma^2\left(u+\frac{\lambda}{\mu}\right)^2\\
	&+2c_2w \left(\gamma v-(\mu+\eta)w\right)+c_2\sigma^2w^2\\
	&+c_3\left(\beta \left(u+\frac{\lambda}{\mu}\right)H(v_t)-(\mu+\gamma+\delta)v\right)+c_4\left(\gamma v-(\mu+\eta)w\right)\\
	&+2(u+v)(-\mu u+\eta w-(\mu\gamma+\delta)v)+\sigma^2\left(\left(u+\frac{\lambda}{\mu}\right)^2+v^2\right)+ \frac{2 c_1 \beta  \lambda ^2}{\mu ^2} \left(v-H(v_t)\right)\\
	=&u^2 (c_1 +1) \left( \sigma ^2 -2 \mu \right)+v^2 \left(\sigma ^2-2 (\gamma +\delta +\mu )\right)+w^2 c_2 \left( \sigma ^2-2 (\eta +\mu )\right)\\
	&+2u w (c_1 \eta + \eta )-2 u v ( \gamma +\delta +2\mu   )+2 v w (\gamma  c_2+ \eta )\\
	&+u \frac{2 (c_1+1) \lambda  \sigma ^2}{\mu }- w c_4  (\eta +\mu )+v (\gamma  c_4-c_3 (\gamma +\delta +\mu ))\\
	&+H(v_t)  \left(u \left(c_3\beta  -\frac{2 c_1 \beta   \lambda }{\mu }\right)-2 c_1\beta   u^2+\frac{ c_3 \beta  \lambda }{\mu }\right)+\frac{(c_1+1) \lambda ^2 \sigma ^2}{\mu ^2}+ \frac{2 c_1 \beta  \lambda ^2}{\mu ^2} \left(v-H(v_t)\right).
	\end{align*}
	By setting $ c_3=\dfrac{2 c_1  \lambda }{\mu} $ and noticing the term $ -2 H(v_t)c_1 \beta   u^2 \leq 0 $, we obtain
	\begin{align*}
	LV\leq &u^2 (c_1 +1) \left( \sigma ^2 -2 \mu \right)+v^2 \left(\sigma ^2-2 (\gamma +\delta +\mu )\right)+w^2 c_2 \left( \sigma ^2-2 (\eta +\mu )\right)\\
		&+2u w (c_1 \eta + \eta )-2 u v ( \gamma +\delta +2\mu   )+2 v w (\gamma  c_2+ \eta )\\
	&+u \frac{2 (c_1+1) \lambda  \sigma ^2}{\mu }-w c_4  (\eta +\mu )+v \left(\frac{2 c_1 \beta  \lambda ^2}{\mu ^2}-\frac{2 c_1 \lambda  (\gamma +\delta +\mu )}{\mu }+\gamma  c_4\right)+\frac{(c_1+1) \lambda ^2 \sigma ^2}{\mu ^2}.
	\end{align*}
	Then we use the inequality $ 2ab\leq a^2+b^2  $ to estimate the cross terms, the above inequality can be rewritten as
	\begin{align*}
	LV\leq &u^2\left(\gamma+\delta+\sigma ^2+\eta+ (\sigma ^2+\eta-2 \mu)c_1 \right)\\
	&+v^2 \left(\gamma  (c_2-1)-\delta +\eta +\sigma ^2\right)\\
	&+w^2 \left((c_1+2) \eta +c_2 \left(\gamma -2 (\eta +\mu )+\sigma ^2\right)\right)\\
	&+u \frac{2 (c_1+1) \lambda  \sigma ^2}{\mu }-w c_4  (\eta +\mu )+v \left(\frac{2 c_1 \beta  \lambda ^2}{\mu ^2}-\frac{2 c_1 \lambda  (\gamma +\delta +\mu )}{\mu }+\gamma  c_4\right)+\frac{(c_1+1) \lambda ^2 \sigma ^2}{\mu ^2}.
	\end{align*}
	We fix $$ c_2= \dfrac{\gamma +\delta -\eta -\sigma ^2}{\gamma }>0,\quad c_4=\dfrac{2 c_1\lambda (\mu+\gamma+\delta)  \left(1-\mathcal{R}_0\right)}{\mu (\mu+\gamma +\eta  )}>0, $$ which are positive by our assumption, such that the coefficient of $ v^2 $ is zero and the coefficients of $ v $ and $ w $ are the same. Therefore,
	\begin{align*}
	LV\leq &u^2\left(\gamma+\delta+\sigma ^2+\eta+ (\sigma ^2+\eta-2 \mu)c_1   \right)\\
	&+w^2 \left(\frac{\gamma ^2+\gamma  ((c_1-1) \eta +\delta -2 \mu )-\left(\delta -\eta -\sigma ^2\right) \left(2 (\eta +\mu )-\sigma ^2\right)}{\gamma }\right)\\
	&+u \frac{2 (c_1+1) \lambda  \sigma ^2}{\mu }-(w+v) \dfrac{2 c_1\lambda(\eta+\mu) (\mu+\gamma+\delta)  \left(1-\mathcal{R}_0\right)}{\mu (\mu+\gamma +\eta  )}+\frac{(c_1+1) \lambda ^2 \sigma ^2}{\mu ^2}.  
	\end{align*}
	Finally, we set 
	$$ c_1:= 2\dfrac{\gamma +\delta +\eta +\sigma ^2}{2 \mu-\eta -\sigma ^2},$$
	then by our assumption $ \mu>\gamma +\delta +\frac{3}{2}(\eta +\sigma ^2)$,  we obtain
	\[ c_1>0\; \text{ and } \; c_1-1<0. \]
	Together with our assumption $ \mu >\dfrac{\gamma ^2+\gamma  \delta -\left(2 \eta -\sigma ^2\right) \left(\delta -\eta -\sigma ^2\right)}{2 \left(\gamma +\delta -\eta -\sigma ^2\right)} $, the coefficient of $ w^2 $ can be estimated as follows	
	\begin{align*}
	&\frac{\gamma ^2+\gamma  ((c_1-1) \eta +\delta -2 \mu )-\left(\delta -\eta -\sigma ^2\right) \left(2 (\eta +\mu )-\sigma ^2\right)}{\gamma }\\
	\leq&  \frac{\gamma ^2+\gamma  (\delta -2 \mu )-\left(\delta -\eta -\sigma ^2\right) \left(2 (\eta +\mu )-\sigma ^2\right)}{\gamma }\leq 0.
	\end{align*}
	Hence, we can see that 
	\begin{align*}
	LV\leq& -u^2\left(\gamma+\delta+\sigma ^2+\eta\right)-(w+v) \dfrac{2 c_1\lambda(\eta+\mu) (\mu+\gamma+\delta)  \left(1-\mathcal{R}_0\right)}{\mu (\mu+\gamma +\eta  )}\\
	&+u \frac{2 (c_1+1) \lambda  \sigma ^2}{\mu }+\frac{(c_1+1) \lambda ^2 \sigma ^2}{\mu ^2}.  
	\end{align*}
	Thus we obtain	
	\begin{equation}\label{eq5.5}
	\begin{aligned}
	\mathrm{d}V\leq &\bigg(-u^2\left(\gamma+\delta+\sigma ^2+\eta\right)-(w+v) \dfrac{2 c_1\lambda(\eta+\mu) (\mu+\gamma+\delta)  \left(1-\mathcal{R}_0\right)}{\mu (\mu+\gamma +\eta  )}\\
	&+u \frac{2 (c_1+1) \lambda  \sigma ^2}{\mu }+\frac{(c_1+1) \lambda ^2 \sigma ^2}{\mu ^2}\bigg)\mathrm{d}t\\
	&+\sigma  \left(2(c_1+1) u^2 +2 v^2+2 c_1 w^2 +4 u v+ u\frac{2 (c_1+1) \lambda }{\mu }+v \left(c_3+\frac{2 \lambda }{\mu }\right)+c_4 w\right)\mathrm{d}W(t).
	\end{aligned}
	\end{equation}
	Integrating both sides of \eqref{eq5.5} from $ 0 $ to $ t $ and then taking the expectation, we obtain
	\begin{align*}
	0\leq& E\left[ V(u(t),v(t),w(t))\right]\\
	\leq& E\left[ V(u(0),v(0),w(0))\right]\\
	&+E\int_{0}^{t}\left(-u(s)^2\left(\gamma+\delta+\sigma ^2+\eta\right)-(w(s)+v(s)) \dfrac{2 c_1\lambda(\eta+\mu) (\mu+\gamma+\delta)  \left(1-\mathcal{R}_0\right)}{\mu (\mu+\gamma +\eta  )}\right.\\
	&+\left.u(s) \frac{2 (c_1+1) \lambda  \sigma ^2}{\mu }+\frac{(c_1+1) \lambda ^2 \sigma ^2}{\mu ^2}\right)\mathrm{d}s.
	\end{align*}
	We divide both sides by $ 2c_1 $ and note that
	\[ \frac{\gamma+\delta+\sigma ^2+\eta}{2c_1}=\dfrac{2\mu-\eta-\sigma^2}{4}, \]
	and	recall the definition of $ K $ where  $$ K=\min\left\lbrace \dfrac{2\mu-\eta-\sigma^2}{4},\dfrac{\lambda(\mu+\eta)(\mu+\gamma+\delta)}{\mu+\gamma+\eta}(1-\mathcal{R}_0) \right\rbrace, $$ these yield	
	 \begin{align*}
	  \limsup_{t \to \infty}&\frac{1}{t}\int_{0}^{t} E\left(u(s)^2+w(s)+v(s)\right)\mathrm{d}s\leq\limsup_{t \to \infty}\frac{1}{t}\int_{0}^{t}\left( \frac{(c_1+1) \lambda  \sigma ^2}{c_1 K \mu  }Eu(s)+ \frac{ \lambda ^2 \sigma ^2}{2K \mu ^2}\frac{c_1+1}{c_1}\right) \mathrm{d}s .
	 \end{align*}
	 Note that $ Eu(t)=E[S(t)-\frac{\lambda}{\mu}]\leq E[\tilde{N}(t)-\frac{\lambda}{\mu}]$, where $ \tilde{N}(t) $ is  the solution of \eqref{Ntilde} and by the property of geometric Brownian motions, one has
	 \begin{equation*}
	 \lim_{t\to \infty} E\tilde{N}(t)=\frac{\lambda}{\mu}.
	 \end{equation*}
	 Thus,
	 \begin{equation*}
	 \limsup_{t \to \infty}\frac{1}{t}\int_{0}^{t} Eu(s)\mathrm{d}s\leq \limsup_{t \to \infty}\frac{1}{t}\int_{0}^{t} E[\tilde{N}(s)-\frac{\lambda}{\mu}]\mathrm{d}s= \lim_{t\to \infty} E\tilde{N}(t)-\frac{\lambda}{\mu}=0.
	 \end{equation*}
	 Therefore, we can see that
	 \begin{align*}
	 \limsup_{t \to \infty}\frac{1}{t}\int_{0}^{t}  Eu^2(s)+Ev(s)+Ew(s)\mathrm{d}s
	 \leq  \frac{ \lambda ^2 \sigma ^2}{2K \mu ^2}\frac{c_1+1}{c_1},
	 \end{align*}
	 which is equivalent to
	  \begin{align*}
	 	\limsup_{t\to \infty}\frac{1}{t}\int_{0}^{t}E\left[\left(S(s)-\frac{\lambda}{\mu}\right)^2+I(s)+R(s)\right]\mathrm{d}s\leq  \frac{ \lambda ^2 \sigma ^2}{2K \mu ^2}\frac{c_1+1}{c_1}.
	  \end{align*}
\end{proof}

\subsection{Around $ E^*  $ endemic equilibrium}
\begin{theorem}\label{thm-ergodic}
	Let Assumption \ref{ass-A} be satisfied.
	If $ \mathcal{R}_0>1 $, $ \mu S^*-\eta R^*>0 $ and $ \sigma $ is small enough, then system \eqref{eq: SFDE} has a unique invariant measure and it is ergodic.
\end{theorem}
\begin{proof}
	Let us recall the endemic equilibrium $ E^*=(S^*,I^*,R^*) $ in \eqref{eq5.2} which yields
	\[ \lambda=\mu S^*-\eta R^*+\beta S^* I^*,\; \mu+\gamma+\delta=\beta S^*,\;\gamma I^*=(\mu +\eta )R^*,  \]
	thus we can rewrite our system \eqref{eq: SFDE} as 
	\begin{equation}
	\begin{cases}
	\mathrm{d}S(t)=\Big[-\mu (S(t)-S^*)+\eta (R(t)-R^*)+\beta \left(S^*I^*-S(t)\int_0^\tau f(s)I(t-s)\mathrm{d}s\right)\Big]\mathrm{d}t+ \sigma S(t)\mathrm{d}W(t),\\[0.25cm]
	\mathrm{d}I(t)=\Big[\beta  S(t)\int_0^\tau f(s)I(t-s)\mathrm{d}s-\beta S^* I(t)\Big]\mathrm{d}t+\sigma I(t)\mathrm{d}W(t),\\[0.25cm]
	\mathrm{d}R(t)=\Big[\gamma (I(t)-I^*)-(\mu + \eta )(R(t)-R^*)\Big]\mathrm{d}t+\sigma R(t)\mathrm{d}W(t).
	\end{cases}
	\end{equation}
	Now we set \[ V_1(S,I)=S^*g\left(\frac{S}{S^*}\right)+I^*g\left(\frac{I}{I^*}\right), \] where \[ g(x)=x-\ln x -1\geq 0,\;\forall x>0. \]
	We calculate $L \left\lbrace S^*g\left(\frac{S(t)}{S^*}\right) \right\rbrace $ and $L \left\lbrace I^*g\left(\frac{I(t)}{I^*}\right) \right\rbrace  $ separately, where
	\begin{equation}\label{eq5.7}
	\begin{aligned}
	 L \left\lbrace S^*g\left(\frac{S(t)}{S^*}\right) \right\rbrace=& \left(1-\frac{S^*}{S(t)}\right) \bigg[-\mu (S(t)-S^*)+\eta (R(t)-R^*)\\
	 &+\beta \left(S^*I^*-S(t)\int_0^\tau f(s)I(t-s)\mathrm{d}s\right)\bigg]+\frac{\sigma^2}{2}S^*\\
	 =&-\mu \frac{(S(t)-S^*)^2}{S(t)}+\eta \left(1-\frac{S^*}{S(t)}\right)\left(R(t)-R^*\right)\\
	 &+ \beta S^*I^* \int_0^\tau f(s)\left(1-\frac{S^*}{S(t)}\right)\left(1-\frac{S(t)}{S^*}\frac{I(t-s)}{I^*}\right)\mathrm{d}s +\frac{\sigma^2}{2}S^*,
	\end{aligned}
	\end{equation}
	and 
	\begin{equation}\label{eq5.8}
	\begin{aligned}
	L \left\lbrace I^*g\left(\frac{I(t)}{I^*}\right) \right\rbrace=& \left(1-\frac{I^*}{I(t)}\right) \bigg[\beta  S(t)\int_0^\tau f(s)I(t-s)\mathrm{d}s-\beta S^* I(t)\bigg]+\frac{\sigma^2}{2}I^*\\
	=&\beta S^*I^* \int_0^\tau f(s)\left(1-\frac{I^*}{I(t)}\right)\left(\frac{S(t)}{S^*}\frac{I(t-s)}{I^*}-\frac{I(t)}{I^*}\right)\mathrm{d}s +\frac{\sigma^2}{2}I^*.
	\end{aligned}
	\end{equation}
	Therefore, incorporating \eqref{eq5.7} and \eqref{eq5.8} we deduce
	\begin{equation*}
	\begin{aligned}
	LV_1=&-\mu \frac{(S(t)-S^*)^2}{S(t)}+\eta \left(1-\frac{S^*}{S(t)}\right)\left(R(t)-R^*\right)+\frac{\sigma
		^2}{2}(S^*+I^*)\\
	&+\beta S^*I^* \int_0^\tau f(s)\left(\mathbf{I}(t,s)+\mathbf{II}(t,s)\right)\mathrm{d}s,
	\end{aligned}
	\end{equation*}
	where 
	\[ \mathbf{I}(t,s)=\left(1-\frac{S^*}{S(t)}\right)\left(1-\frac{S(t)}{S^*}\frac{I(t-s)}{I^*}\right),\; \mathbf{II}(t,s)=\left(1-\frac{I^*}{I(t)}\right)\left(\frac{S(t)}{S^*}\frac{I(t-s)}{I^*}-\frac{I(t)}{I^*}\right). \]
	We claim that for any $ t>0 $ and $ s\in [0,\tau] $,
	\[ \mathbf{I}(t,s)+\mathbf{II}(t,s)\leq 0. \]\
	In fact, by simple calculation, we have
	\begin{align*}
	\mathbf{I}(t,s)+\mathbf{II}(t,s)=&2-\frac{S^*}{S(t)}+\frac{I(t-s)}{I^*}-\frac{S(t)}{S^*}\frac{I(t-s)}{I(t)}-\frac{I(t)}{I^*}\\
	=&-g\left(\frac{S^*}{S(t)}\right)-g\left(\frac{S(t)}{S^*}\frac{I(t-s)}{I(t)}\right)-\left(g\left(\frac{I(t)}{I^*}\right)-g\left(\frac{I(t-s)}{I^*}\right)\right).
	\end{align*}
	Since $ g(x)\geq 0,\,\forall x>0 $. In addition,
	\[ g\left(\frac{I(t)}{I^*}\right)-g\left(\frac{I(t-s)}{I^*}\right)=\frac{I(t)-I(t-s)}{I^*}\big[\ln I(t-s)-\ln I(t)\big] \leq 0, \]
	where we used the inequality $ (x-y)(\ln x-\ln y ) \geq 0,\;\forall x,y>0 $.
	Therefore, we can see that
	\begin{equation}\label{eq5.9}
	LV_1\leq -\mu \frac{(S(t)-S^*)^2}{S(t)}+\eta \left(1-\frac{S^*}{S(t)}\right)\left(R(t)-R^*\right)+\frac{\sigma
		^2}{2}(S^*+I^*).
	\end{equation}
	Now we consider the non-negative function \[ V_2(R)=\dfrac{\eta}{\gamma S^*} \dfrac{(R-R^*)^2}{2}, \]
	and by defining $ N(t)=S(t)+I(t)+R(t),\; N^*=S^*+I^*+R^* $, we can calculate 
	\begin{equation}\label{eq5.10}
	\begin{aligned}
	LV_2=& \frac{\eta}{\gamma S^*} (R(t)-R^*) \Big[\gamma (I(t)-I^*)-(\mu + \eta )(R(t)-R^*)\Big] +\frac{\eta}{\gamma S^*}\frac{\sigma^2}{2}R^2\\
	=&\frac{\eta}{\gamma S^*} (R(t)-R^*) \Big[\gamma (N(t)-S(t)-R(t)-(N^*-S^*-I^*))-(\mu + \eta )(R(t)-R^*)\Big] \\
	&+\frac{\eta}{\gamma S^*}\frac{\sigma^2}{2}(R(t)-R^*+R^*)^2\\
	\leq &\frac{\eta}{\gamma S^*} (R(t)-R^*) \Big[\gamma (N(t)-N^*)-\gamma (S(t)-S^*)-(\gamma+\mu  +\eta )(R(t)-R^*)\Big] \\
	&+\frac{\eta}{\gamma S^*}\sigma^2\big[(R(t)-R^*)^2+(R^*)^2\big]\\
	=&\frac{\eta}{ S^*} \big[(R(t)-R^*) (N(t)-N^*)- (R(t)-R^*) (S(t)-S^*)\big]\\
	&-\frac{\eta}{ \gamma S^*} (\gamma+\mu+\eta-\sigma^2)(R(t)-R^*)^2+\frac{\eta}{\gamma S^*}\sigma^2(R^*)^2.
	\end{aligned}
	\end{equation}
	Finally, we consider the non-negative function \[ V_3(R,N)=\dfrac{\eta\gamma}{\delta (2\mu+\eta) S^*}\frac{1}{2}\left(N-N^*+\frac{\delta}{\gamma}(R-R^*)\right)^2. \]
	We can calculate that
	\begin{equation*}
	\begin{aligned}
		&L \left\lbrace \left(N(t)-N^*+\frac{\delta}{\gamma}(R-R^*)\right)^2 \right\rbrace\\
		=& \bigg(N(t)-N^*+\frac{\delta}{\gamma}(R-R^*)\bigg)\bigg(\lambda-\mu N(t)-\delta I(t)+\frac{\delta}{\gamma}\left(\gamma I(t)-(\mu+\eta)R(t)\right) \bigg)\\
		&+\frac{\sigma^2}{2}\bigg(N(t)+\frac{\delta}{\gamma}R(t)\bigg)^2\\
		 \leq&\bigg(N(t)-N^*+\frac{\delta}{\gamma}(R-R^*)\bigg) \bigg(\lambda-\mu N(t)-\frac{\delta(\mu+\eta)}{\gamma}R(t) \bigg) +{\sigma^2}\bigg(N(t)^2+\frac{\delta^2}{\gamma^2}R(t)^2\bigg).
	\end{aligned}	
	\end{equation*}
	Since we have $ \lambda=\mu N^*+ \dfrac{\delta(\mu+\eta)}{\gamma}R^* $, therefore we can rewrite 
	\begin{equation*}
	\begin{aligned}
	&L \left\lbrace \left(N(t)-N^*+\frac{\delta}{\gamma}(R-R^*)\right)^2 \right\rbrace\\
	\leq&\bigg(N(t)-N^*+\frac{\delta}{\gamma}(R-R^*)\bigg) \bigg(-\mu (N(t)-N^*)-\frac{\delta(\mu+\eta)}{\gamma}(R(t) -R^*)\bigg) +{\sigma^2}\bigg(N(t)^2+\frac{\delta^2}{\gamma^2}R(t)^2\bigg)\\
	=&-\mu \left(N(t)-N^*\right)^2-\frac{\delta(2\mu+\eta)}{\gamma}\left(N(t)-N^*\right)\left(R(t)-R^*\right)-\frac{\delta^2(\mu+\eta)}{\gamma^2}(R(t) -R^*)^2\\
	&+2{\sigma^2}\bigg((N(t)-N^*)^2+(N^*)^2+\frac{\delta^2}{\gamma^2}\left(R(t)-R^*\right)^2+(R^*)^2\bigg)\\
	=&-\left(\mu-2\sigma^2\right) \left(N(t)-N^*\right)^2-\frac{\delta(2\mu+\eta)}{\gamma}\left(N(t)-N^*\right)\left(R(t)-R^*\right)-\frac{\delta^2}{\gamma^2}(\mu+\eta-2\sigma^2)(R(t) -R^*)^2\\
	&+2{\sigma^2}\bigg((N^*)^2+(R^*)^2\bigg).\\
	\end{aligned}	
	\end{equation*}
	Therefore, we can see that
	\begin{equation}\label{eq5.11}
	\begin{aligned}
	LV_3\leq& -\dfrac{\left(\mu-2\sigma^2\right)\eta\gamma}{\delta (2\mu+\eta) S^*} \left(N(t)-N^*\right)^2-\frac{\eta}{S^*}\left(N(t)-N^*\right)\left(R(t)-R^*\right)-\frac{\delta\eta(\mu+\eta-2\sigma^2)}{\gamma(2\mu+\eta)S^*}(R(t) -R^*)^2\\
	&+\dfrac{2{\sigma^2}\eta\gamma}{\delta (2\mu+\eta) S^*}\bigg((N^*)^2+(R^*)^2\bigg).
	\end{aligned}
	\end{equation}
	Now incorporating equations \eqref{eq5.9}, \eqref{eq5.10}, and \eqref{eq5.11}, we obtain
	\begin{equation}\label{eq5.12}
	\begin{aligned}
	&L\big(V_1+V_2+V_3\big)\\
	\leq& -\mu \frac{(S(t)-S^*)^2}{S(t)}+\eta \left(1-\frac{S^*}{S(t)}\right)\left(R(t)-R^*\right)+\frac{\sigma
		^2}{2}(S^*+I^*)\\
	&+\frac{\eta}{ S^*} \big[ (N(t)-N^*)(R(t)-R^*)-  (S(t)-S^*)(R(t)-R^*)\big]\\
	&-\frac{\eta}{ \gamma S^*} (\gamma+\mu+\eta-\sigma^2)(R(t)-R^*)^2+\frac{\eta}{\gamma S^*}\sigma^2(R^*)^2\\
	&-\dfrac{\left(\mu-2\sigma^2\right)\eta\gamma}{\delta (2\mu+\eta) S^*} \left(N(t)-N^*\right)^2-\frac{\eta}{S^*}\left(N(t)-N^*\right)\left(R(t)-R^*\right)-\frac{\delta\eta(\mu+\eta-2\sigma^2)}{\gamma(2\mu+\eta)S^*}(R(t) -R^*)^2\\
	&+\dfrac{2{\sigma^2}\eta\gamma}{\delta (2\mu+\eta) S^*}\bigg((N^*)^2+(R^*)^2\bigg)\\
	=&-\mu \frac{(S(t)-S^*)^2}{S(t)}+\eta \frac{(S(t)-S^*)(R(t)-R^*)}{S(t)}-\eta\frac{(S(t)-S^*)(R(t)-R^*)}{ S^*}\\
	&-\frac{\eta}{ \gamma S^*} \left(\gamma+\mu+\eta-\sigma^2+\frac{\delta(\mu+\eta-2\sigma^2)}{2\mu+\eta}\right)(R(t)-R^*)^2\\
	&-\dfrac{\left(\mu-2\sigma^2\right)\eta\gamma}{\delta (2\mu+\eta) S^*} \left(N(t)-N^*\right)^2+\sigma^2 \tilde{K}\\
	=&-\left(\mu+\frac{\eta}{S^*}(R(t)-R^*)\right)\frac{(S(t)-S^*)^2}{S(t)}-\frac{\eta}{ \gamma S^*} \left(\gamma+\mu+\eta-\sigma^2+\frac{\delta(\mu+\eta-2\sigma^2)}{2\mu+\eta}\right)(R(t)-R^*)^2\\
	&-\dfrac{\left(\mu-2\sigma^2\right)\eta\gamma}{\delta (2\mu+\eta) S^*} \left(N(t)-N^*\right)^2+\sigma^2 \tilde{K},
	\end{aligned}
	\end{equation}
	where \[ \tilde{K}=\frac{1}{2}(S^*+I^*)+\frac{\eta}{\gamma S^*}(R^*)^2+\dfrac{2\eta\gamma}{\delta (2\mu+\eta) S^*}\bigg((N^*)^2+(R^*)^2\bigg). \]
	Since $ R(t)>0,\,\forall t>0 $, we have
	\[ -\left(\mu+\frac{\eta}{S^*}(R(t)-R^*)\right)\leq -\left(\mu-\frac{\eta}{S^*}R^*\right) \]
	and by our assumption $ \mu S^*-\eta R^*>0 $, if we set $ \sigma^2 \leq \mu /2  $ and define
	\[ \tilde{m}=\min\left\lbrace \mu-\frac{\eta}{S^*}R^*,\frac{\eta}{ \gamma S^*} \left(\gamma+\mu+\eta-\sigma^2+\frac{\delta(\mu+\eta-2\sigma^2)}{2\mu+\eta}\right),\dfrac{\left(\mu-2\sigma^2\right)\eta\gamma}{\delta (2\mu+\eta) S^*} \right\rbrace>0, \]
	from \eqref{eq5.12} we can deduce
	\[ L\big(V_1+V_2+V_3\big)\leq -\tilde{m}\left(\frac{(S(t)-S^*)^2}{S(t)}+(R(t)-R^*)^2+\left(N(t)-N^*\right)^2\right) +\sigma^2 \tilde{K}. \] 
	If we denote the ``cobblestone'' area by
	\[ D_\sigma:= \left\lbrace (S,I,R)\in \R_+^3:\frac{(S-S^*)^2}{S}+(R-R^*)^2+\left(S+I+R-\left(S^*+I^*+R^*\right)\right)^2 \leq \frac{\sigma^2\tilde{K}}{\tilde{m}}  \right\rbrace, \]
	for $ \sigma $ sufficiently small, we have the distance $ \rho\left(D_\sigma,\partial \R^3_+\right) >0 $. Then one can take $ U $ as any neighborhood of the region $ D_\sigma $ such that $ \overline{U}\subset \R^3_+ $, where $ \overline{U} $ is the closure of $ U $. Hence, for some $ \kappa>0 $, $ L(V_1+V_2+V_3) <-\kappa $  for any $ (S,I,R)\in \R^3_+ \backslash U $. This implies that (ii) in Proposition \ref{pro-ergodic} is satisfied. Moreover, Proposition \ref{pro-ergodic} (i) is ensured by Remark \ref{rem2}. As a consequence, the model \eqref{eq: SFDE} has a unique invariant measure and it is ergodic.
\end{proof}
	
\section{Numerical simulations}

In this section, we show simulations with two sets of parameters satisfying the conditions in Theorem \ref{thm disease free} and in Theorem \ref{thm-ergodic} respectively. We adopt the Euler-Maruyama method \cite{Higham} and set system \eqref{eq: SFDE} with $ H(\cdot) $ of discrete delay type $ H(\phi)=\phi(-\tau) $. The corresponding discretized equations are

\begin{equation}\label{eq discrete sirs}
\left\{\begin{array}{ll}
S_{k+1}=S_k+\left(\lambda-\mu S_k-\beta S_kH(I_k)+\eta R_k\right)\Delta t+ \sigma S_k\xi_k\sqrt{\Delta t},\\[0.2cm]
I_{k+1}=I_k+\left(\beta S_kH(I_k)-(\mu+\delta+\gamma) I_k\right)\Delta t+ \sigma I_k\xi_k\sqrt{\Delta t},\\[0.2cm]
R_{k+1}=R_k+\left(\gamma I_k-(\mu+\eta) R_k\right)\Delta t+ \sigma R_k\xi_k\sqrt{\Delta t},
\end{array}	\right.
\end{equation}
where  $ \xi_k,\ k=1,2,\ldots,n, $ are independent Gaussian random variables $ N(0,1) $ and $ \sigma $ is the intensity of randomness. Note that with Assumption \ref{ass-A}, the convergence of the discretized equations can be guaranteed (see \cite{Mao03}).

We fix our parameters as 
\begin{equation}\label{parameters}
\lambda=0.05,\,\mu=0.05,\,\gamma=0.035,\, \delta=0.005,\,\eta=0.002,\,\sigma=0.05,\, \tau =10
\end{equation}
and we  take the initial value $ \xi$ to be a constant function, i.e.,
\begin{equation}\label{eq6.3}
 S(\theta)\equiv0.7,\quad I(\theta)\equiv0.3,\quad  R(\theta)\equiv0,\quad \forall\ \theta \in [-\tau,0], 
\end{equation}
and we simulate the solution to the system \eqref{eq: SFDE} with different values of $\beta $. 

\begin{figure}[htbp]
	\centering
	\includegraphics[width=3.5in]{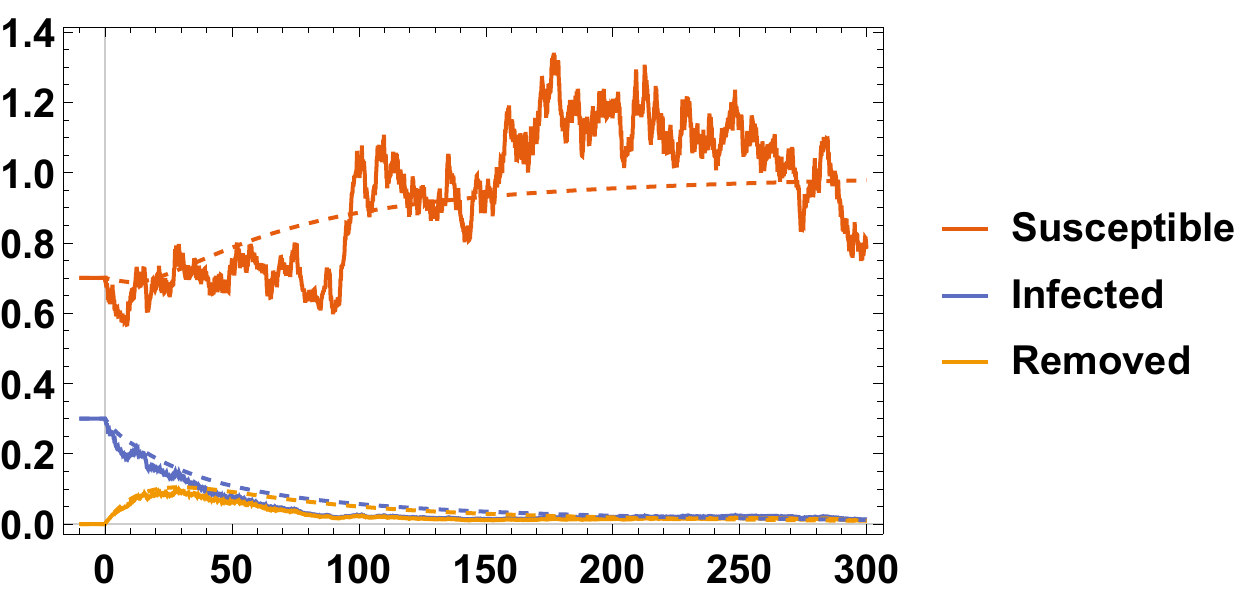}\ 
	\caption{ \textit{The simulation of one path of the solutions of system \eqref{eq: SFDE} up to time $ t=300 $ with the initial value $ \xi$ as in \eqref{eq6.3}. Here $ \beta=0.08,\sigma=0.05 $ and other parameters are from \eqref{parameters}. One can calculate $ \mathcal{R}_0\approx0.8889<1 $ and the conditions in Theorem \ref{thm disease free} are satisfied.  The dashed lines are solutions of the deterministic delay differential equations.}}\label{FIG1}
\end{figure}

In Figure \ref{FIG1}, we set $  \beta=0.08,\,\sigma=0.05 $, thus we can compute $ \mathcal{R}_0\approx0.8889<1 $ and the conditions in Theorem \ref{thm disease free} are satisfied. In the simulation, we use dashed lines and solid lines to compare the solution of the deterministic delay differential equation with one path of the solutions of system \eqref{eq: SFDE}. It is known from \cite{Beretta95} that when $ \mathcal{R}_0<1 $, the disease free equilibrium $ (\lambda/\mu,0,0)=(1,0,0) $ is asymptotically stable.  We can see from the simulation that the solution to \eqref{eq: SFDE} fluctuates around the deterministic solution in a small amplitude, which confirms the conclusion of Theorem \ref{thm disease free}. 

\begin{figure}[htbp]
	\centering
	\includegraphics[width=3.5in]{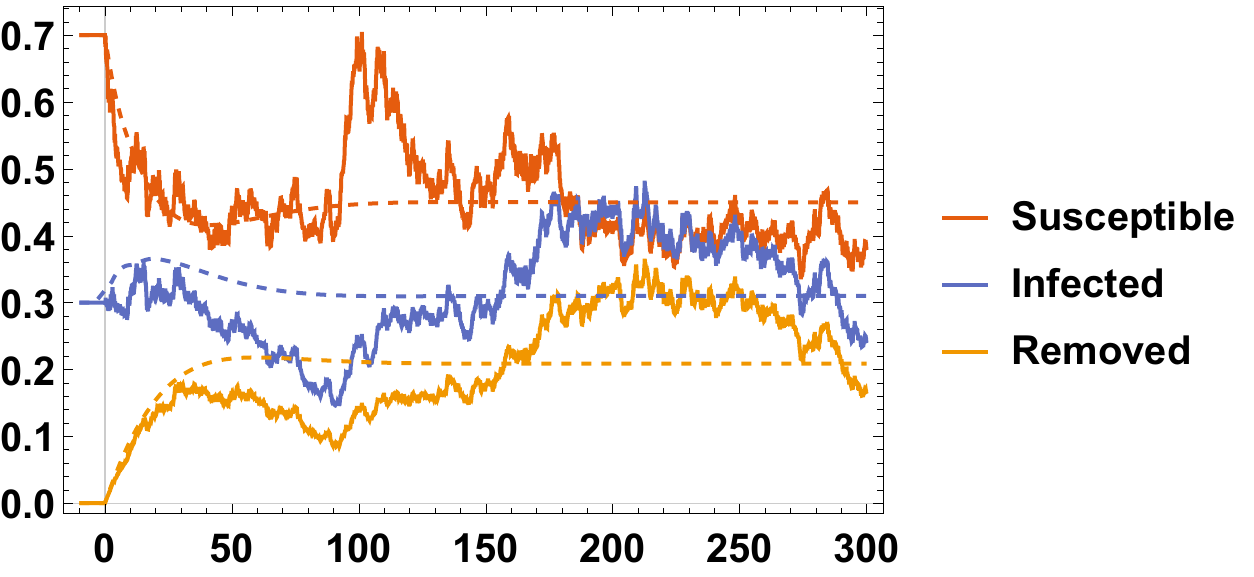}\ 
	\caption{ \textit{The simulation of one path of the solutions of system \eqref{eq: SFDE} up to time $ t=300 $ with the initial value $ \xi$ as in \eqref{eq6.3}. Here we set $ \beta=0.2,\sigma=0.05 $ and other parameters are from \eqref{parameters}. One can calculate $ \mathcal{R}_0\approx2.222>1,\; \mu S^*-\eta R^*\approx 0.022 $, thus the conditions in Theorem \ref{thm-ergodic} are satisfied.  The dashed lines are solutions of the deterministic delay differential equations.}}\label{FIG2}
\end{figure}
\begin{figure}[htbp]
		\begin{center}
				\includegraphics[width=2.4in]{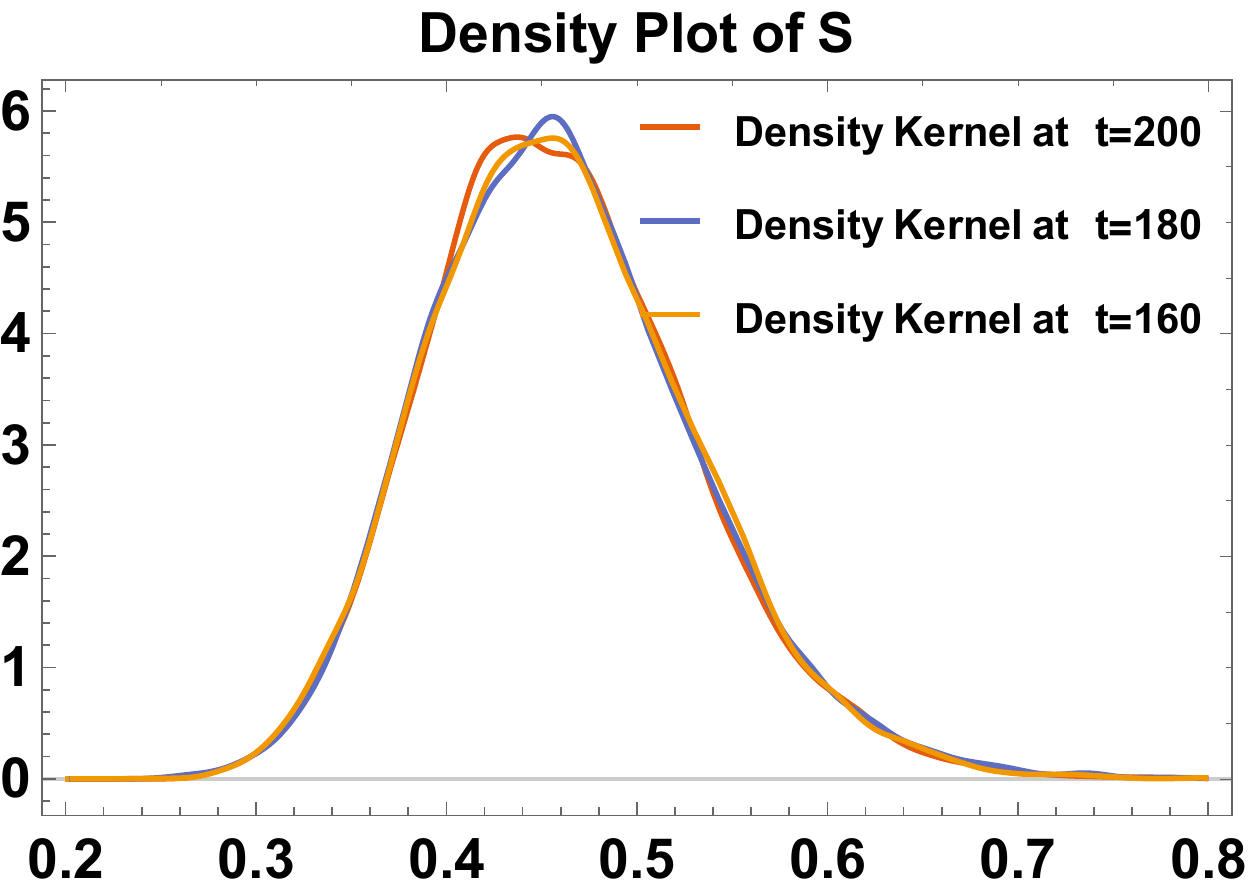}\ \ \includegraphics[width=2.4in]{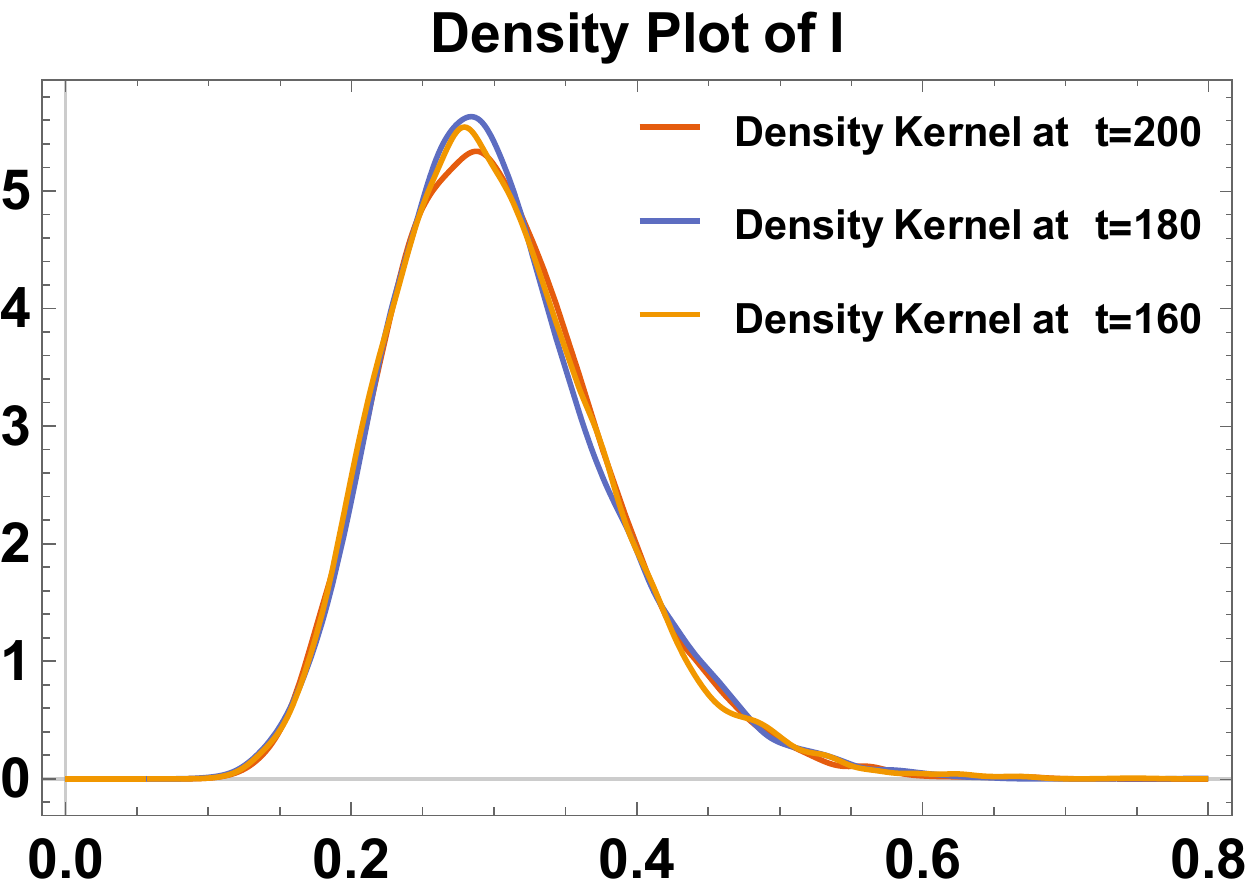}
				\includegraphics[width=2.4in]{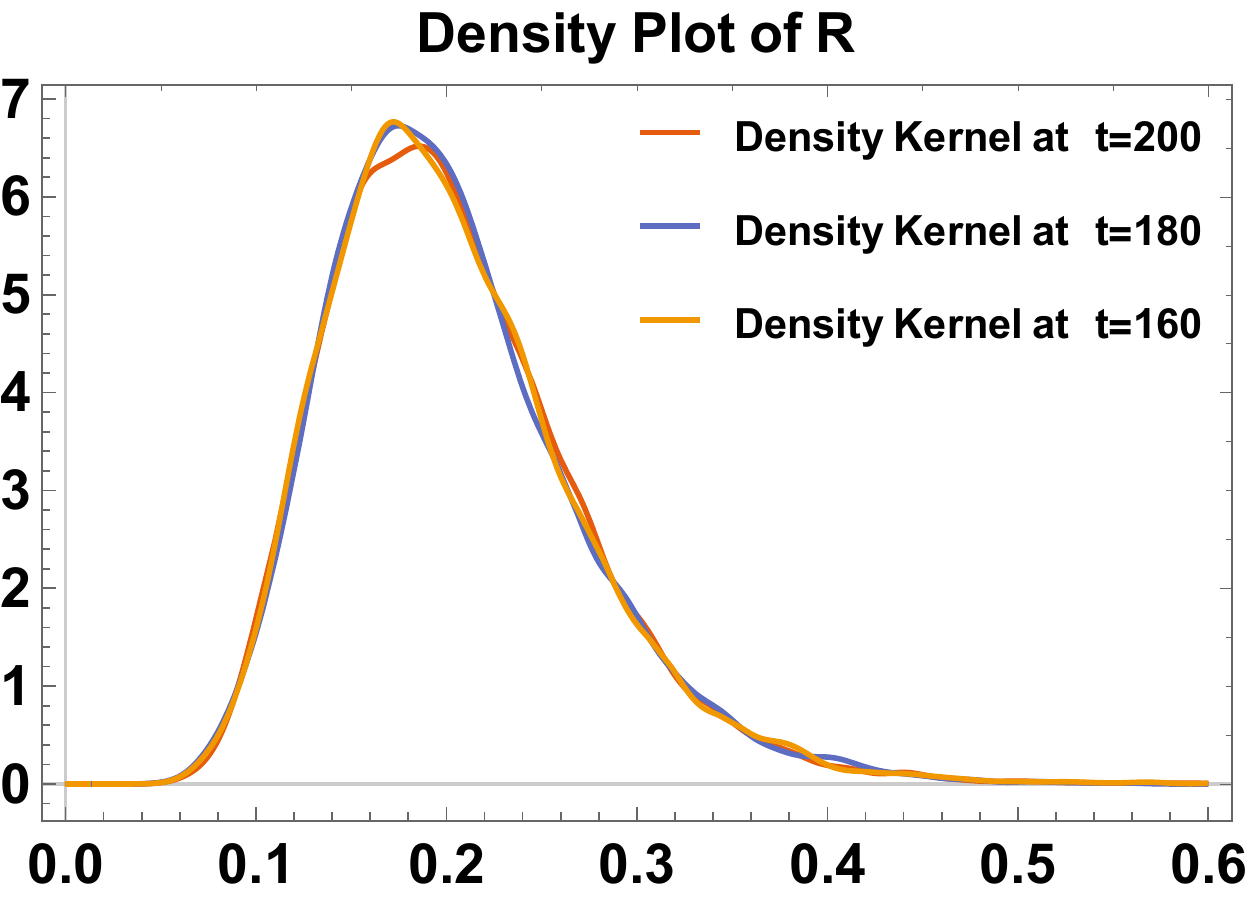}
			\caption{\textit{Density plot based on 10 000 stochastic simulations for each group at time $ t=160,180 $ and $ 200 $. Here we choose $ \beta=0.2 $, $ \sigma=0.05 $ and  other parameters are from \eqref{parameters}. The simulations confirm the existence of the unique ergodic invariant measure for system \eqref{eq: SFDE} }}
			\label{fig:side:b}\label{FIG3}
		\end{center}
	\end{figure}

In Figure \ref{FIG2} and \ref{FIG3}, we set $ \beta=0.2 $ and $ \sigma=0.05 $. In this case, we can compute $ \mathcal{R}_0\approx2.222>1,\ \mu S^*-\eta R^*\approx 0.022 $, thus the conditions in Theorem \ref{thm-ergodic} are satisfied. 
Figure \ref{FIG2} simulates one path of the solutions up to time $ t=300 $ (solid lines) with comparison to the solution of the deterministic delay differential equation (dashed lines).
In Figure \ref{FIG3}, we simulate the density kernels of solutions \eqref{eq: SFDE} with three groups namely $ (S,I,R) $. In the simulation, the density kernels are based on 10 000 sample paths. Our initial values are as in \eqref{eq6.3}.
Comparing these density kernels, we can see that the density plot of each group at different time $ t $ for $ t=160,180,200 $ stay almost the same. Therefore, we can conclude that the simulations strongly indicate the existence of the unique ergodic invariant measure for the system \eqref{eq: SFDE}. 

\section{Conclusion}
In this paper, we studied the existence and ergodicity of the invariant measure for a stochastic delayed SIRS model.  Furthermore, we discussed the asymptotic behavior around disease-free equilibrium when $ \mathcal{R}_0<1 $.  Our Theorem \ref{thm main result 1} suggests that, under a fairly general condition, the existence of the invariant measure can be guaranteed. Moreover for this invariant measure to be unique and ergodic, one sufficient condition is  the noise intensity $ \sigma $ to be sufficiently small (Theorem \ref{thm-ergodic}).  Simulations have been carried out to support our analytical results.

There are still some topics which deserve further research. For example, one can consider the situation where the stochastic noise affects on several parameters in a heterogeneous way and one can also consider the existence of periodic solutions for epidemic models under random perturbations \cite{Wang14,Zu15}.



\begin{thebibliography}{00}
	
	\bibitem{Beretta98}
	E. Beretta, V. Kolmanovskii, L. Shaikhet, Stability of epidemic model with time delays influenced by stochastic perturbations, {\em Math. Comput. Simulat.}, {\bf 45}  (1998),  269-277.
	
	\bibitem{Beretta95}
	E. Beretta, Y. Takeuchi, Global stability of an SIR epidemic model with time delays, {\em J. Math. Biol.}, {\bf 33}  (1995),  250-260.
	
	\bibitem{BP}
	P. Billingsley, {\em Convergence of Probability Measures},  John Wiley \& Sons Inc., New York, 1999.
	
	\bibitem{Cai15}
	Y. L. Cai, Y. Kang, M. Banerjee, W. M. Wang, A stochastic SIRS epidemic model with infectious force under intervention strategies, {\em J. Differential. Equations.}, {\bf 259}  (2015),  7463-7502.
	
	\bibitem{Cai17}
	Y. L. Cai, Y. Kang, W. M. Wang, A stochastic SIRS epidemic model with nonlinear incidence rate. \textit{Appl. Math. Comput.}, \textbf{305} (2017), 221-240.

	\bibitem{DPG;ZJ}
	G. Da Prato, J. Zabczyk, {\em Ergodicity for Infinite Dimensional Systems},  Cambridge University Press, 1996.
	
	\bibitem{Es-Sarhir10}
	A. Es-Sarhir, M. Scheutzow, O. Van Gaans, Invariant measures for stochastic functional differential equations with superlinear drift term, {\em Differential Integral Equations}, {\bf 23}  (2010),  189--200.
	
	\bibitem{Enastu12}
	Y. Enatsu,  Y. Nakata, Y. Muroya, Lyapunov functional techniques for the global stability analysis of a delayed SIRS epidemic model, \textit{Nonlinear Anal. Real World Appl.}, \textbf{13(5)} (2012), 2120-2133.
	
	\bibitem{Gard88}
	T. Gard, \textit{Introduction to Stochastic Differential Equations}, New York, 1988.
	
	\bibitem{Grey11}
	 A. Gray,  D. Greenhalgh,  L. Hu, X. Mao, J. Pan, A stochastic differential equation SIS epidemic model, {\em SIAM J. Appl. Math.}, {\bf 71}  (2011),  876--902.
	
	\bibitem{Hattaf} 
	K. Hattaf, M. Mahrouf, J. Adnani, N. Yousfi, Qualitative analysis of a stochastic epidemic model with specific functional response and temporary immunity,  {\em Physica. A.}, \textbf{490} (2018), 591-600. 
	 
	\bibitem{Higham}
	 D. J. Higham, An algorithmic introduction to numerical simulation of stochastic differential equations, {\em SIAM. Rev.}, {\bf 43}  (2001),  525-546.
	
	\bibitem{Ikeda}
	N. Ikeda, S. Watanabe,  A comparison theorem for solutions of stochastic differential equations and its applications. \textit{Osaka J. Math.}, \textbf{ 14(3)} (1977), 619-633.
	
	\bibitem{ImhofWalcher05}
	L. Imhof,  S. Walcher, Exclusion and persistence in deterministic and stochastic chemostat models, {\em J. Differential. Equations.}, {\bf 217}  (2005),  26-53.
	
	\bibitem{Jiang11}
	D. Q. Jiang, J. J. Yu, C. Y. Ji,  N. Z. Shi, Asymptotic behavior of global positive solution to a stochastic SIR model, {\em  Math. Comput. Modelling}, {\bf 54}  (2011),  221-232.
	
	\bibitem{karatzas12}
	I. Karatzas, S. E. Shreve, {\em Brownian Motion and Stochastic Calculus},  Springer-Verlag, New York, 1991.
	
	
	
	\bibitem{Kha11}
	R. Khasminskii  \textit{Stochastic stability of differential equations} (Vol. 66). Springer Science \& Business Media, 2011.
	
	\bibitem{kinnally09}
	M. S. Kinnally, {\em Stationary distributions for stochastic delay differential equations with non-negativity constraints}, Thesis (Ph.D.)--University of California, San Diego. 124 pp.  ProQuest LLC, Ann Arbor, MI, 2009.
	
	\bibitem{KW10}
	M. Kinnally,  R. Williams, On existence and uniqueness of stationary distributions for stochastic delay differential equations with positivity constraints. \textit{Electron. J. Probab.}, \textbf{15} (2010), 409-451.
	
	\bibitem{Lahrouz13}
	 A. Lahrouz,  L. Omari, Extinction and stationary distribution of a stochastic SIRS epidemic model with non-linear incidence, \textit{Statist. Probab. Lett.}, \textbf{83(4)} (2013), 960-968.
	
	\bibitem{Li15}
	 D. Li, J. A. Cui,  M. Liu, S. Liu, The evolutionary dynamics of stochastic epidemic model with nonlinear incidence rate, \textit{Bull. Math. Biol.}, \textbf{77(9)} (2015), 1705-1743.
	
	\bibitem{Liu15}  Q. Liu, Q. Chen, Analysis of the deterministic and stochastic SIRS epidemic models with nonlinear incidence, {\em Physica. A.}, \textbf{428} (2015), 140-153.


	\bibitem{Liu17}
	 Q. Liu, D. Jiang, N. Shi, T. Hayat, A. Alsaedi, Stationary distribution and extinction of a stochastic SIRS epidemic model with standard incidence, {\em Physica. A.}, \textbf{469} (2017), 510-517.	
	 
	 \bibitem{Liu17_2}
	 Q. Liu, D. Jiang, N. Shi, T. Hayat, A. Alsaedi, Asymptotic behavior of stochastic multi-group epidemic models with distributed delays. {\em Physica. A.}, \textbf{467} (2017), 527-541.
	 
	 \bibitem{Liu17_3}
	 Q. Liu, D. Jiang, N. Shi, T. Hayat, A. Alsaedi,  Stationarity and periodicity of positive solutions to stochastic SEIR epidemic models with distributed delay, \textit{Discrete Contin. Dyn. Syst. Ser. B}, \textbf{22(6)} (2017), 2479-2500.	
	
	\bibitem{Lu09}
	 Q. Lu, Stability of SIRS system with random perturbations, {\em Physica. A.}, {\bf 388}  (2009),  3677--3686.
	
	\bibitem{Mao03}
	 X. Mao, Numerical solutions of stochastic functional differential equations, {\em LMS J. Comput. Math.}, {\bf 6}  (2003),  141--161 (electronic).
	
	\bibitem{Maoxuerong}
	 X. Mao, {\em Stochastic Differential Equations and Applications},  Elsevier, 2007.
	
	\bibitem{McCluskey10}
	 C. C. McCluskey, Complete global stability for an SIR epidemic model with delay - Distributed or discrete, {\em Nonlinear Anal. Real World Appl.}, {\bf 11}  (2010),  55-59.

	
	\bibitem{Nakata11} Y. Nakata, Y. Enatsu, Y. Muroya, On the global stability of an SIRS epidemic model with distributed delays. \textit{Discrete Contin. Dyn. Syst.}, Suppl. Vol. II (2011), 1119-1128.
	
	\bibitem{Rudnicki03}
	R. Rudnicki, Long-time behaviour of a stochastic prey-predator model, {\em Stoch. Proc. Appl.}, {\bf 108}  (2003),  93-107.
	

	\bibitem{Scheutzow13}
	M. Scheutzow, A stochastic Gronwall lemma, {\em Infin. Dimens. Anal. Quantum Probab. Relat. Top.}, {\bf 16}  (2013),  1350019, 1350014.

	\bibitem{Tornatore05}
	 E. Tornatore,  S. M. Buccellato,  P. Vetro, Stability of a stochastic SIR system, {\em Physica. A.}, {\bf 354}  (2005),  111-126.
	
	
	\bibitem{Wang14}
	F. Wang, X. Wang, S. Zhang, C. Ding, . On pulse vaccine strategy in a periodic stochastic SIR epidemic model. \textit{Chaos Solitons Fractals}, \textbf{66} (2014), 127-135.
	
	\bibitem{Yang12}
	 Q. S. Yang,  D. Q. Jiang, N. Z. Shi, C. Y. Ji, The ergodicity and extinction of stochastically perturbed SIR and SEIR epidemic models with saturated incidence, {\em J. Math. Anal. Appl.}, {\bf 388}  (2012),  248-271.
	
	\bibitem{Yang14}
	Q. Yang, X. Mao, Stochastic dynamics of SIRS epidemic models with random perturbation, \textit{Math. Biosci. Eng.} \textbf{11(4)} (2014), 1003-1025.
	
	\bibitem{Zhao15}
	Y. Zhao, D. Jiang, X. Mao, A. Gray, The threshold of a stochastic SIRS epidemic model in a population with varying size, \textit{Discrete Contin. Dyn. Syst. Ser. B}, \textbf{20(4)} (2015), 1277-1295.
		
	\bibitem{ZMH06} J. Zhen, Z. Ma, M. Han . Global stability of an SIRS epidemic model with delays. \textit{ Acta Math. Sci. Ser. B}, \textbf{26(2)} (2006), 291-306.
	
	\bibitem{ZhuYin07}
	C. Zhu,  G. Yin, Asymptotic properties of hybrid diffusion systems. \textit{SIAM J. Control Optim.}, \textbf{46(4)} (2007), 1155-1179.
	
	\bibitem{Zu15}
	L. Zu,  D. Jiang, D. O'Regan, B. Ge, Periodic solution for a non-autonomous Lotka--Volterra predator-prey model with random perturbation. \textit{J. Math. Anal. Appl.}, \textbf{430(1)} (2015), 428-437.
	\end{thebibliography}
\end{document}